\documentclass[reqno]{amsart}
\usepackage{graphicx}

\usepackage{amssymb}
\usepackage{amsfonts}
\usepackage{indentfirst}
\usepackage{graphicx, graphics, psfig, pstricks}
\usepackage{pst-plot}

\theoremstyle{plain}
  \newtheorem{thm}{Theorem}[section]
  \newtheorem{prop}[thm]{Proposition}
  \newtheorem{lem}[thm]{Lemma}
  \newtheorem{cor}[thm]{Corollary}
  
  \newtheorem*{cth*}{Theorem}
  
  \newtheorem*{clm*}{Lemma}
  
\theoremstyle{definition}
  
  \newtheorem*{defn*}{Definition}
  
\theoremstyle{remark}
  \newtheorem{rem}{Remark}

\begin{document}
\title[]{On minimal decomposition of $p$-adic polynomial dynamical systems}
\author{Aihua Fan}
\address{LAMFA UMR 6140, CNRS,
Universit\'e de Picardie Jules Verne, 33, Rue Saint Leu, 80039
Amiens Cedex 1, France
} \email{ai-hua.fan@u-picardie.fr}

\author{Lingmin Liao}
\address{LAMFA UMR 8050, CNRS
Universit\'e Paris-Est Cr\'eteil Val de Marne, 61 Avenue du
G\'en\'eral de Gaulle, 94010 Cr\'eteil Cedex, France
 \& Department of Mathematics, Wuhan
University, 430072 Wuhan, China}\email{lingmin.liao@univ-paris12.fr}
\subjclass[2000]{Primary 37E99; Secondary 11S85, 37A99}
\keywords{$p$-adic dynamical system, minimal component, quadratic
polynomial} \maketitle

\begin{abstract}
A polynomial of degree $\ge 2$ with coefficients in the ring of $p$-adic numbers $\mathbb{Z}_p$
is studied as a dynamical system on $\mathbb{Z}_p$.
It is proved that
 the dynamical behavior of such a system is totally described by its
  minimal subsystems. For an arbitrary  quadratic polynomial on
 $\mathbb{Z}_2$, we exhibit all its minimal subsystems.
\end{abstract}

\markboth{On the minimal decomposition of $p$-adic polynomial dynamical systems}{}%

\section{Introduction}
Let $\mathbb{Z}_p$ be the ring of $p$-adic integers ($p$ being a
prime number).  Let $f\in {\mathbb Z}_p[x]$ be a polynomial of
coefficients in $\mathbb{Z}_p$ and with degree $\deg f \ge 2$. It is
simple fact that $f: \mathbb{Z}_p \to \mathbb{Z}_p$ is a
$1$-Lipschitz map. In this paper we study the
 topological dynamical system $(\mathbb{Z}_p, f)$. We refer to \cite{walters}
 for dynamical terminology and \cite{Kob, Mahler, sch,Serre} for notions related to $p$-adic numbers.

Our first theorem is a general result which shows that a polynomial system admits at most countably
many minimal subsystems. This   describes to some extent  the dynamical behavior of the
system (see Theorem \ref{thm-decomposition}).

\medskip

\noindent{\bf Theorem A.} {\em  Let $f \in \mathbb{Z}_p[x]$ with
$\deg f \ge 2$. We have the following decomposition
$$
     \mathbb{Z}_p = A \bigsqcup B \bigsqcup C
$$
where $A$ is the finite set consisting of all periodic points of
$f$, $B= \bigsqcup_i B_i$ is the union of all (at most countably
many) clopen invariant sets such that each $B_i$ is a finite union
of balls and each subsystem $f: B_i \to B_i$ is minimal, and each
point in $C$ lies in the attracting basin of a periodic orbit or of
a minimal subsystem. }

\medskip

 We will refer to the above decomposition as {\em the minimal
decomposition} of the system $f : \mathbb{Z}_p \to \mathbb{Z}_p$.
A finite periodic orbit of $f$ is by definition a minimal set.
But for the convenience of the present paper, only the sets $B_i$
in the above decomposition are called {\em minimal components}.

Recently the theory of Non-Archimedean, in particular of $p$-adic,
dynamical systems has been intensively developed
(\cite{Benedetto1998, Benedetto2001a, Benedetto2001b}, \cite{DSV},
\cite{DKM}, \cite{FLWZ}, \cite{GKL2001}, \cite{HY1983},
\cite{Hsia1996, Hsia2000}, \cite{KN2001}, \cite{KLPS2007},
\cite{Lubin}, \cite{Rivera-Letelier2003}, \cite{TVW1989},
\cite{WS1998}). See also the monographes
\cite{Anashin-Khrennikov09}, \cite{KN}, \cite{Silverman} and the
bibliographies therein.

 There were few works done on the minimal decomposition. Multiplications on
$\mathbb{Z}_p$ ($p\ge 3$) were studied by Coelho and Parry
\cite{coparry} and general affine maps were studied by Fan, Li, Yao
and Zhou \cite{FLYZ}. The minimal decomposition of a polynomial system has been known in these cases and only in these cases.
Quadratic polynomials will be studied at the end of the present paper.

One of interesting problems well studied in the literature is the
minimality of the system $f: \mathbb{Z}_p \to \mathbb{Z}_p$, which
corresponds to the situation where $A=C=\emptyset$ and $B$ consists
of one minimal component (\cite{Anashin94, Anashin98, Anashin1,
Anashin2, CFF,DZ, K1991, Khrennikov2003, khr, KN, Knuth, Larin}).

 The above
theorem shows that there are only a finite number of periodic
orbits. The possible periods are shown in the following theorem (see
Theorem \ref{possible-periods}). The statements 1)-3) were known to
Pezda \cite{Pezda}, the statements 1) and 2) are also found by
Desjardins and Zieve \cite{DZ} in a different way. The statement 4)
is new.

\medskip
\noindent {\bf Theorem B.} {\em  Let $f \in \mathbb{Z}_p[x]$.\\
\indent \mbox{\rm 1)} \  If $p\ge 5$, the periods of periodic orbits
are of the form $a b$
with $a| (p-1)$ and $1\le b \le p$.\\
\indent \mbox{\rm 2)} \  If $p=3$, the periods of periodic orbits
must be
$1,2,3,4,6$ or $9$.\\
\indent \mbox{\rm 3)} \  If $p=2$, the periods of periodic orbits
must be $1,2$ or $4$.}\\
\indent \mbox{\rm 4)} \ Let $p=2$. If there is $4$-periodic orbit,
then $f(z \!\!\!\mod 2) \!\!\!\mod 2$ should be a permutation on $\mathbb{Z}/2\mathbb{Z}$.
There is no $4$-periodic orbit for quadratic polynomials.

\medskip

What kind of set can be a minimal component of a polynomial system?
In a recent work, Chabert, Fan and Fares \cite{CFF} showed that each
minimal component $B_i$ must be a Legendre set and that in general,
any Legendre set is a minimal component of  some $1$-Lipschtz
system. We will show that for a polynomial system, the minimal
components $B_i$ are Legendre sets of special forms. Let
$(p_s)_{s\ge 1}$ be a sequence of positive integers such that
$p_s|p_{s+1}$ for every $s\ge 1$. We denote by $\mathbb{Z}_{(p_s)}$
the inverse limit of $\mathbb{Z}/p_s \mathbb{Z}$, which is called
 an odometer. The map $x \to x+1$ is called the adding machine on
$\mathbb{Z}_{(p_s)}$. We will prove the following theorem (see
Theorem \ref{structure-minimal}).

\medskip
\noindent{\bf Theorem C.} {\em  Let $f \in \mathbb{Z}_p[x]$ with
$\deg f \ge 2$. If $E$ is a minimal clopen invariant set of $f$,
then $f : E \to E$ is conjugate to the adding machine on an odometer
$\mathbb{Z}_{(p_s)}$, where  $$(p_s) = (k, kd, k dp, kdp^2,
\cdots)$$ with integers $k$ and $d$ such that $1 \leq k\leq p$ and $d|
(p-1)$.}
\medskip

As we have already pointed out, the minimal decomposition is fully studied for affine maps.
It seems  much more difficult  to study the minimal decomposition for higher order polynomials.
In this paper, we  try to attack the
problem for quadratic polynomials.  For an arbitrary $2$-adic quadratic polynomial $$ f(x) = ax^2 + bx
+c$$
 on $\mathbb{Z}_2$, we find all its minimal components.

As we shall see,
such a quadratic system $f: \mathbb{Z}_2 \to \mathbb{Z}_2$ is conjugate to one
of the following quadratic polynomials $$ x^2-\lambda, \quad x^2+bx,
\quad x^2+x-d$$
 where $\lambda \in \mathbb{Z}_2$, $b\equiv 1 \ ({\rm mod}\ 2)$ and $\sqrt{d} \not\in \mathbb{Z}_2$.
   Our results are
 stated in  Theorems \ref{x^2-lambda}-\ref{x^2+x-d-3}. Let us state here some of these
 results.

\medskip
\noindent{\bf Theorem D.}
{\em  Consider the polynomial $x^2-\lambda$ on $\mathbb{Z}_2$.\\
 \indent {\rm 1)} If $\lambda \equiv 0 \ ({\rm mod} \ 4) $, then there are
 two attracting fixed points, one in  $4\mathbb{Z}_2 $ with basin $2\mathbb{Z}_2 $, and the other one in $1+
 4\mathbb{Z}_2$ with basin $1+2\mathbb{Z}_2$.\\
 \indent {\rm 2)} If $\lambda \equiv 1 \ ({\rm mod} \ 4) $, then the whole
 $\mathbb{Z}_2 $ is attracted into a periodic orbit of period $2$ with one orbit point in  $4\mathbb{Z}_2 $, and the other one in $3+
 4\mathbb{Z}_2$.\\
 \indent {\rm 3)} If $\lambda \equiv 2 \ ({\rm mod} \ 4) $, then there are
 two attracting fixed points, one in  $2+4\mathbb{Z}_2 $ with basin $2\mathbb{Z}_2 $, and the other one in $3+
 4\mathbb{Z}_2$ with basin $1+2\mathbb{Z}_2 $.\\
 \indent {\rm 4)} If $\lambda \equiv 3 \ ({\rm mod} \ 4) $, then the whole
 $\mathbb{Z}_2 $ is attracted into a periodic orbit of period $2$ with one orbit point in  $1+4\mathbb{Z}_2 $, and the other one in $2+
 4\mathbb{Z}_2$.}
\medskip

\medskip
\noindent{\bf Theorem E.} {\em Consider the polynomial $f(x) =
x^2+x$ on $\mathbb{Z}_2[x]$.  There is  one fixed point $0$. We have
 $f(1+2\mathbb{Z}_2) \subset 2\mathbb{Z}_2$ and we can decompose $2\mathbb{Z}_2$
 as
\[2\mathbb{Z}_2=\{0\}\bigsqcup
\left(\bigsqcup_{n\geq 2} 2^{n-1} +2^n\mathbb{Z}_2 \right). \] Each
$2^{n-1} +2^n\mathbb{Z}_2$ ($n\geq 2$) consists of $2^{n-2}$ pieces
of minimal components:
\[2^{n-1} +t2^n +2^{2n-2}\mathbb{Z}_2, \quad t=0, \dots , 2^{n-2}-1.\]
}

\medskip
\noindent{\bf Theorem F.} {\em Consider the polynomial
$f(x)=x^2+x-d$ with $d\equiv 3 \ ({\rm mod} \ 4)$.  Then
$f(2\mathbb{Z}_2)
 \subset 1+2\mathbb{Z}_2$ and  $1+2\mathbb{Z}_2$ is the unique
 minimal component of $f$.
}
\medskip


The main idea used in the paper comes from Desjardins and Zieve's
work \cite{DZ} and the Ph.D thesis of Zieve \cite{Zieve}. Let $E$ be
an $f$-invariant compact set. It is now well known  that the
subsystem $(E, f)$ is minimal if and only if the induced map $f_n:
E/p^n\mathbb{Z} \to E/p^n\mathbb{Z}$ is minimal (transitive) for any
$n\ge 1$ (see \cite{Anashin2,CFF}). The idea of Desjardins and Zieve
is to establish relations between $f_n$'s cycles and $f_{n+1}$'s
cycles, by linearizing the $k$-th iteration $f_{n+1}^k $ on a cycle
of $f_n$ of length $k$.

The paper is organized as follows. In Section \ref{preliminary}, we
give a full development of the idea in \cite{DZ} by studying the
induced dynamical systems $f_n$ on $\mathbb{Z}/p^n\mathbb{Z} $ when
$p\geq 3 $. Section \ref{top-2-adic} is devoted to the case of $p=2$
which was not treated in \cite{DZ}. As we shall see, the situation
in the case $p=2$ is not exactly the same as in the case $p\ge 3$.
In Sections \ref{formation} and \ref{typicalversion}, we investigate
how  a minimal component is formed  by analyzing  the reduced maps
$f_n$ ($n\geq 1 $) and we prove the decomposition theorem.  In
Section \ref{typicalversion}, we discuss the possible forms of
minimal components. In Section \ref{2-adic}, we give a detailed
description of the minimal decomposition for  an arbitrary quadratic
polynomial system on $ \mathbb{Z}_2$.

\section{Induced dynamics on $\mathbb{Z}/p^n \mathbb{Z}$ ($p\ge 3$)} \label{preliminary}
The main core of this section follows Desjardins and Zieve
\cite{DZ}. We shall give more details and rewrite some proofs for
reader's convenience. The case $p=2$, which is a little bit special,
will be fully discussed in the next section.

Let $p\ge 3$ be a prime (we may replace $3$ by $2$ in many places).
Let $n\ge 1$ be a positive integer. Denote by $f_n$ the induced
mapping of $f$ on $\mathbb{Z}/p^n\mathbb{Z}$, i.e.,
$$f_n (x \ {\rm mod} \ p^{n})=f (x)  \ \quad \ \ {\rm mod} \ p^{n}.$$
 Many properties of the dynamics $f$ are linked to those of
$f_n$. One is the following.

\begin{thm}[\cite{Anashin2}, \cite{CFF}]\label{minimal-part-to-whole}
Let $f\in \mathbb{Z}_p[x]$ and $E\subset \mathbb{Z}_p$ be a compact
$f$-invariant set. Then $f: E\to E$ is minimal if and only if $f_n:
E/p^n\mathbb{Z}_p \to E/p^n\mathbb{Z}_p$ is minimal for each $n\ge
1$.
\end{thm}

It is clear that if $f_n: E/p^n\mathbb{Z}_p \to E/p^n\mathbb{Z}_p$
is minimal, then $f_m: E/p^m\mathbb{Z}_p \to E/p^m\mathbb{Z}_p$ is
also minimal for each $1\le  m <n$. So, the above theorem shows that
it is important to investigate under what condition, the minimality
of $f_{n}$ implies that of $f_{n+1}$.

Assume that $\sigma=(x_1, \cdots, x_k) \subset
\mathbb{Z}/p^n\mathbb{Z}$ is a cycle of $f_n$ of length $k$ (also
called $k$-cycle), i.e.,
$$ f_n(x_1)=x_2, \cdots, f_n(x_i)=x_{i+1}, \cdots, f_n(x_k)=x_1.$$
In this case we also say $\sigma$ is at level $n$. Let
$$ X:=\bigsqcup_{i=1}^k  X_i \ \ \mbox{\rm where}\ \
X_i:=\{ x_i +p^nt; \ t=0, \cdots,p-1\} \subset
\mathbb{Z}/p^{n+1}\mathbb{Z}.$$
 Then
\[
f_{n+1}(X_i) \subset X_{i+1}  \ (1\leq i \leq k-1) \ \ \mbox{\rm
and}\ \  f_{n+1}(X_k) \subset X_1.\]


In the following we shall study the behavior of the finite dynamics
$f_{n+1}$ on the $f_{n+1}$-invariant set $X$ and determine all
cycles in $X$ of $f_{n+1}$, which will be  called  {\it lifts} of
$\sigma$ (from level $n$ to level $n+1$). Remark that the length of any lift $\tilde{\sigma}$ of
$\sigma$ is a multiple of $k$.

 Let $g:=f^k$ be the $k$-th iterate of $f$. Then, any point in $\sigma$
 is fixed by $g_n$, the $n$-th induced map of $g$.
 For $x\in \sigma$, denote
\begin{eqnarray}
& &a_n(x):=g'(x)=\prod_{j=0}^{k-1} f'(f^j(x)) \label{def-an} \\
& &b_n(x):=\frac{g(x)-x}{p^n}=\frac{f^k(x)-x}{p^n}.\label{def-bn}
\end{eqnarray}
The values on the cycle $\sigma=(x_1,\dots,x_k )$ of the functions
$a_n$ and $b_n$ are important for our purpose. They define, for each $x$, an
affine map
$$\Phi(x,t)=b_n(x)+a_n(x) t \qquad (x \in \sigma, t \in \mathbb{Z}/p \mathbb{Z}).$$
The $1$-order Taylor expansion of $g$ at $x$ implies
\begin{eqnarray}\label{linearization}
 g(x+p^n t) \equiv x+p^n b_n(x) + p^n a_n(x) t
 \equiv x + p^n \Phi(x,t) \quad ({\rm mod} \ p^{2n}).
 \end{eqnarray}
An important consequence of the last formula shows that $g_{n+1}:
X_i \to X_i$ is conjugate to the linear map $$ \Phi(x_i, \cdot):
\mathbb{Z}/p\mathbb{Z} \to \mathbb{Z}/p\mathbb{Z}.
$$
We could call it  the {\em linearization} of $g_{n+1}: X_i \to X_i$.

For any $x\in X$ (and even for any $x\in \mathbb{Z}_p$), we can
define the values of $a_n(x)$ and $b_n(x)$ by the formulas
(\ref{def-an}) and (\ref{def-bn}). As we shall see in the following
lemma, the coefficient $a_n(x)$ (mod $p$) is always constant on
$X_i$ and the coefficient $b_n(x)$ (mod $p$) is also constant on
$X_i$ but under the condition $a_n(x)\equiv 1$ (mod $p$).

Denote by $v_p(n)$ the $p$-valuation of $n$.

\begin{lem}\label{lem1}
Let $n\geq 1 $ and $\sigma=(x_1, \cdots, x_k)$ be a $k$-cycle of $f_n$. \\
 \indent {\rm (i)}\ For $1\leq i,j\leq k$, we have
$$a_n(x_i)\equiv a_n(x_j) \ \ ({\rm mod} \ p^n).$$
\indent {\rm (ii)}\ For for $1\leq i\leq k$ and $0\le t\le p-1$, we
have $$ a_n(x_i+p^nt)\equiv a_n(x_i) \ \ ({\rm mod} \ p^n).$$
\indent {\rm (iii)}\ For $1\leq i\leq k$ and $0\le t\le p-1$, we
have
\[ b_n(x_i+p^n t) \equiv b_n(x_i) \quad ({\rm mod} \ p^A),
\]
where
 $A:=\min \{ v_p (a_n(x_i)-1),n \}=\min \{ v_p (a_n(x_j)-1),n
\}$ for $1\leq i,j\leq k$.
\\
 \indent {\rm (iv)}\  For all $1\le i, j\le k$ we have
 $$\min
\{v_p(b_n(x_i)),A\}= \min \{v_p(b_n(x_j)),A \}.$$ Consequently, if
$a_n(x_i) \equiv 1 \ ({\rm mod} \ p^n)$,
$$\min \{v_p(b_n(x_i)),n \}=
\min \{v_p(b_n(x_j)),n \}. $$
\end{lem}

\begin{proof}
 Assertion {\rm (i)} follows directly from the definition of
$a_n(x_i)$ and the fact that $\sigma = (x_i, f_n(x_i), \cdots,
f_n^{k-1}(x_i))$.  The assertion (ii) is a direct consequence of
\[
a_n(x_i+p^nt) \equiv  \prod_{j=1}^k f'(f^j(x_i+p^nt)) \equiv
\prod_{j=1}^k f'(f^j(x_i))  \quad ({\rm mod} \ p^n) .
\]

The $1$-order Taylor expansion of $g$ at $x_i$ gives
\begin{align*}
g(x_i+p^n t) -(x_i+p^n t) 
 \equiv p^n\left(\frac{g(x_i)-x_i}{p^n}\right) +p^n t(g'(x_i)-1) \quad ({\rm mod} \ p^{2n}).
\end{align*}
Hence
\[ b_{n}(x_i+p^n t)\equiv b_n(x_i) + t (a_n(x_i)-1) \quad ({\rm mod} \ p^{n}). \]
Then (iii) follows.

Write
$$
g(f(x_i))-f(x_i) = f(f^k(x_i))-f(x_i) = f(x_i+p^n b_n(x_i)) -f(x_i).
$$
The $1$-order  Taylor expansion $f$ at $x_i$ leads to $$
g(f(x_i))-f(x_i) \equiv p^n b_n(x_i) f'(x_i) \quad ({\rm mod} \
p^{2n}).
$$
Hence we have
\[b_{n}(f(x_i))\equiv b_n(x_i) f'(x_i) \quad ({\rm mod} \ p^{n}).
\]
Since when $A= 0$, the result is obvious, we may suppose that $ A
\neq 0$. Then $a_n(x_i) \equiv 1 \ ({\rm mod} \ p)$ (for $1\leq i
\leq k$) which implies $f'(x_i) \not\equiv 0 \ ({\rm mod} \ p)$ for
all $1\leq i \leq k$. Notice that $f(x_i) \equiv x_{i+1} \ ({\rm
mod} \ p^n) $. Then by (iii), we obtain (iv).
\end{proof}

\medskip
According to  Lemma \ref{lem1} (i) and (ii),  the value of $a_n(x) \
({\rm mod} \ p^{n}) $ does not depend on  $x\in X$. According to
Lemma \ref{lem1} (iii) and (iv), whether $b_n(x) \equiv 0 \ ({\rm
mod} \ p)$ does not depend on  $x\in X$ if $a_n(x) \equiv 1 \ ({\rm
mod} \ p)$. For simplicity, sometimes we shall write $a_n$ and $b_n$
without mentioning $x$.

\medskip
The above analysis allows us to
distinguish the following four behaviors  of $f_{n+1}$ on $X$:\\
 \indent {\rm (a)} If $a_n \equiv 1 \ ({\rm mod} \ p)$ and
 $b_n \not\equiv 0 \ ({\rm mod} \ p)$, then $\Phi$ preserves a single cycle of length $p$, so that
 $f_{n+1}$ restricted to $X$ preserves a single cycle of length $pk$. In this case we say $\sigma$ {\it grows}.\\
 \indent {\rm (b)} If $a_n \equiv 1 \ ({\rm mod} \ p)$ and
 $b_n \equiv 0 \ ({\rm mod} \ p)$, then $\Phi$ is the identity, so $f_{n+1}$ restricted to $X$
 preserves
 $p$ cycles  of length $k$. In this case we say $\sigma$ {\it splits}. \\
 \indent {\rm (c)} If $a_n \equiv 0 \ ({\rm mod} \ p)$, then $\Phi$ is constant, so $f_{n+1}$ restricted to $X$
 preserves
 one cycle of length $k$ and  the remaining points of $X$ are mapped into this cycle.
  In this case we say $\sigma$ {\it grows tails}. \\
  \indent {\rm (d)} If $a_n \not\equiv0, 1 \ ({\rm mod} \ p)$, then $\Phi$ is a permutation
  and the $\ell$-th iterate of $\Phi$ reads
 \begin{eqnarray*}
\Phi^\ell(x, t) = b_n(a_n^{\ell}-1) /(a_n-1) +a_n^{\ell} t
\end{eqnarray*}
so that $$ \Phi^{\ell}(t)-t =(a_n^{\ell}-1) \left( t+
\frac{b_n}{a_n-1}\right).$$
 Thus,
$\Phi$ admits a single fixed point $t=-b_n/(a_n-1)$, and the
remaining points  lie on cycles of length $d$, where $d$ is the
order of $a_n$ in $(\mathbb{Z}/p\mathbb{Z})^*$. So, $f_{n+1}$
restricted to $X$ preserves one cycle of length $k$ and
$\frac{p-1}{d}$ cycles of length $kd$. In this case we say $\sigma$
{\it partially splits}.
\medskip

Now let us study the relation between $(a_n, b_n)$ and $(a_{n+1},
b_{n+1})$. Our aim is to see the change of nature from a cycle to
its lifts.

\begin{lem}\label{anbn}
Let $\sigma=(x_1,\dots,x_k)$ be a $k$-cycle of $f_n$and let
$\tilde{\sigma}$ be a lift of $\sigma$ of length $kr$, where $r\geq
1$ is an integer.  We have
\begin{eqnarray}\label{an}
a_{n+1}(x_i+p^nt) \equiv a_n^r(x_i) \quad ({\rm mod} \ p^{n}),
\qquad (1\leq i \leq k, 0\le t \le p-1)
\end{eqnarray}
\begin{equation}\label{bn}
  \begin{split}
 & pb_{n+1}(x_i+p^n t)\\
\equiv& t(a_n(x_i)^r-1)  + b_n(x_i)\big(1+a_n(x_i)+ \cdots +
a_n(x_i)^{r-1}\big)  \quad ({\rm mod} \ p^{n}).
\end{split}
\end{equation}
\end{lem}

\begin{proof}
 The formula (\ref{an}) follows from
\[
a_{n+1} \equiv (g^r)'(x_i+p^n t) \equiv (g^r)'(x_i) \equiv
\prod_{j=0}^{r-1} g'(g^j(x_i))
 \equiv a_n^r \ ({\rm mod} \ p^{n}).
 \]

 By repeating $r$ times of the linearization (\ref{linearization}), we obtain
 $$
 g^r(x_i+p^nt) \equiv x_i + \Phi^r(x_i,t) p^n \quad ({\rm mod} \
 p^{2n}),
 $$
 where $\Phi^r$ means the $r$-th composition of $\Phi$ as function
 of $t$, and
$$
\Phi^r(x_i,t) =  ta_n(x_i)^r + b_n(x_i)
 \big(1+a_n(x_i)+ \cdots + a_n(x_i)^{r-1}\big).
$$
Thus (\ref{bn}) follows from
  the definition of $b_{n+1}$ and the above two expressions.
\end{proof}
\medskip

By Lemma \ref{anbn}, we obtain immediately the following
proposition.
\begin{prop}
Let $n\geq 1$. Let $\sigma$ be a $k$-cycle of $f_n$
and $\tilde{\sigma}$ be a lift of $\sigma$. Then we have\\
 \indent {\rm 1)} if $a_n \equiv 1 \ ({\rm mod} \ p)$, then $a_{n+1} \equiv 1 \ ({\rm mod} \ p)$;\\
 \indent {\rm 2)} if $a_n \equiv 0 \ ({\rm mod} \ p)$, then $a_{n+1} \equiv 0 \ ({\rm mod} \ p)$;\\
 \indent {\rm 3)} if $a_n \not\equiv 0,1 \ ({\rm mod} \ p)$ and $\tilde{\sigma}$ is of length
 $k$, then $a_{n+1} \not\equiv 0,1 \ ({\rm mod} \ p)$;\\
 \indent {\rm 4)} if $a_n \not\equiv 0,1 \ ({\rm mod} \ p)$ and $\tilde{\sigma}$ is of length
 $kd$ where $d\ge 2$ is the order of $a_n$ in $(\mathbb{Z}/p\mathbb{Z})^*$, then $a_{n+1} \equiv 1 \ ({\rm mod} \ p)$.
\end{prop}

This result is interpreted as follows in dynamical system language:\\
 \indent {\rm 1)} If $\sigma$ grows or splits, then any lift $\tilde{\sigma}$ grows
or splits. \\
 \indent {\rm 2)}
If $\sigma$ grows tails, then the single lift $\tilde{\sigma}$
also grows tails. \\
 \indent {\rm 3)}
 If $\sigma$ partially splits, then the
 lift $\tilde{\sigma}$
 of the same length as $\sigma$ partially splits, and the other lifts  of length $kd$ grow or split.

\medskip

If $\sigma=(x_1, \cdots, x_k)$ is a cycle of $f_n$ which grows
tails, then $f$ admits a $k$-periodic point $x_0$ in the clopen set
$\mathbb{X} =\bigsqcup_{i=1}^k (x_i +p^n \mathbb{Z}_p)$ and
$\mathbb{X}$ is contained in the attracting basin of the periodic
orbit $x_0, f(x_0), \cdots, f^{k-1}(x_0)$.
%


%

With the preceding preparations, we are ready to prove  the
following Propositions \ref{cycle-p>3}-\ref{partiallysplits} which
predict the behavior of the lifts of a cycle $\sigma$ by the
properties of $\sigma$. We refer the reader to \cite{DZ} for their
proofs. Otherwise we can follow the similar proofs  of Propositions
\ref{grows}-\ref{weaklysplits} in  the case $p=2$.

\begin{prop}[\cite{DZ}]\label{cycle-p>3}
Let $\sigma$ be a growing cycle of $f_n$ and $\tilde{\sigma}$
be the unique lift of $\sigma$. \\
 \indent {\rm 1)} If $ p\geq 3 $ and $n\geq 2$ then $\tilde{\sigma}$ grows.\\
 \indent {\rm 2)} If $p>3$ and $n\geq 1$ then $\tilde{\sigma}$ grows. \\
 \indent {\rm 3)} If $p=3$ and $n=1$, then $\tilde{\sigma}$ grows if and only if
 $b_1(x) \not\equiv g''(x)/2 \ ({\rm mod} \ p)$.
\end{prop}

According to 1) and 2) of Proposition \ref{cycle-p>3}, in the cases
$ p\geq 3, n\geq 2$ and $p>3, n\geq 1$, if $\sigma =(x_1,
\cdots,x_k)$ grows then its lift also grows, and the lift of the
lift will  grow and so on. So, the clopen set
$$\mathbb{X}=\bigsqcup_{i=1}^k (x_i + p^n \mathbb{Z}_p)$$ is a minimal
set by Theorem \ref{minimal-part-to-whole}.

Let $$A_n(x):= v_p(a_n(x)-1), \quad B_n(x):= v_p(b_n(x)).$$
 By Lemma
\ref{lem1}, for a cycle $\sigma=(x_1,\dots,x_k)$, $\min
\{A_n(x_i),n\} $ does not depend on the choice of $x_i, 1\leq i \leq
k$ and if $B_n(x_i) < \min \{A_n(x_i), n \} $ then $B_n(x_i)$ does
not depend on the choice of $x_i, 1\leq i \leq k$. Sometimes, there
is no difference when we choose $x_i$ or $x_j$ in the cycle. So,
without misunderstanding, we will not mention $x_i$ in $A_n$ and
$B_n$ (see the proof of Proposition \ref{stronglysplits} for the
details corresponding to the case $p=2$).

We say that a cycle $\sigma$ at level $n$ {\em splits $\ell$ times} if
$\sigma $ splits, and the lifts of $ \sigma$ at level $n+1$ split and
inductively all lifts at level $n+j$ ($2\leq j < \ell$) split.
Similarly, one can imagine what we  mean if we say a cycle {\em grows $\ell$
times}. That a cycle {\em grows forever} means that it grows infinite times.

\begin{prop}[\cite{DZ}]\label{splits-p3} Let $p\geq 3$ and $n\geq 1$.
Let $\sigma$ be a splitting cycle of $f_n$. \\
 \indent {\rm  1)} If $\min \{A_n, n\} > B_n $,  every lift splits $B_n-1$
 times then all  lifts at level $n + B_n$ grow forever.\\
 \indent {\rm  2)} If $A_n\leq B_n$ and $A_n< n$,  there is one lift which behaves the same as
 $\sigma$ (i.e., this lift splits and $A_{n+1}\leq B_{n+1}$ and $A_{n+1}< {n+1}$) and other lifts split $A_n-1$ times then  all  lifts at level $n + A_n$ grow forever..\\
 \indent {\rm 3)} If $B_n \geq n$ and $A_n \geq n$, then all lifts split at least $n-1$
 times.
\end{prop}

\begin{prop}[\cite{DZ}]\label{partiallysplits}
Let $p\geq 3$ and $n\geq 1$. Let $\sigma$ be a partially splitting
$k$-cycle of $f_n$  and  $\tilde{\sigma}$ be a lift of $\sigma$ of
length $kd$, where $d$ is the order of
 $a_n$ in $\mathbb{Z}/p\mathbb{Z}$. \\
 \indent {\rm 1)} If $A_{n+1} < nd $, then $\tilde{\sigma}$ splits $A_{n+1}-1$
 times then all  lifts at level $n+A_{n+1}$  grow forever.\\
 \indent {\rm 2)} If $A_{n+1} \geq nd $, then $\tilde{\sigma}$ splits at least $nd-1$
 times.
\end{prop}

We remark that in the partially splitting case, $\min \{ A_{n+1}(x),
nd\}$ depends only on the lifting cycle of $f_{n+1}$ of length $kd$
but not on $x$ (see \cite{DZ}, Corollary 3).
\medskip

\setcounter{equation}{0}

\section{Induced dynamics on $\mathbb{Z}/p^n \mathbb{Z}_p$ ($p= 2$)
}\label{top-2-adic}
 In this section we
focus on the special case $p=2$ which is not considered in
\cite{DZ}. The first part in the preceding  section (where $p\ge 3 $
is not explicitly assumed) remains true for $p=2$. Notice that when
$p=2$, there is no partially splitting cycle.

We only need to study how a cycle grow or split.
 We distinguish four cases. Let $\sigma$ be a cycle of $f_n$.
We say $\sigma$ {\it strongly grows} if $a_n \equiv 1 \ ({\rm mod} \
4)$ and $b_n \equiv 1 \ ({\rm mod} \ 2)$,  and $\sigma$ {\it weakly
grows} if $a_n \equiv 3 \ ({\rm mod} \ 4)$ and $b_n \equiv 1 \ ({\rm
mod} \ 2)$. We say $\sigma$ {\it strongly splits} if $a_n \equiv 1 \
({\rm mod} \ 4)$ and $b_n \equiv 0 \ ({\rm mod} \ 2)$, and $\sigma$
{\it weakly splits} if $a_n \equiv 3 \ ({\rm mod} \ 4)$ and $b_n
\equiv 0 \ ({\rm mod} \ 2)$.

The following results hold true when $p=2$. Their proofs are
postponed and got together at the end of this section.

\begin{prop}\label{grows}
Let $\sigma$ be a cycle of $f_n$ ($n\geq 2$). If $\sigma$ strongly
grows then the lift of ${\sigma}$ strongly grows. If $\sigma$ weakly
grows then the lift of ${\sigma}$ strongly splits.
\end{prop}

The first assertion of Proposition \ref{grows}  implies that if
$\sigma=(x_1, \cdots,x_k)$ is a strongly growing cycle of $f_n$
($n\ge 2$), then $\bigsqcup (x_i+ p^n \mathbb{Z}_p)$ is a minimal
set.

Recall that $$ A_n(x)= v_2(a_n(x)-1), \quad B_n(x)= v_2(b_n(x)).
$$
In the following proposition, the $x$ in $A_n(x), B_n(x)$ can be
chosen any $x_i$ of the cycle $\sigma=(x_1, \cdots,x_k)$ (see its
proof).
\begin{prop}\label{stronglysplits}
Let $\sigma$ be a strongly splitting cycle of $f_n$ ($n\geq 2$). \\
 \indent {\rm  1)} If $\min \{A_n, n\} > B_n $, then all lifts strongly split $B_n-1$
 times, then all the lifts at level $n+B_n$ strongly grow.\\
 \indent {\rm  2)} If $A_n\leq B_n$ and $A_n< n$, then one lift behaves the same as
 $\sigma$ (i.e., this lift strongly splits and $A_{n+1}\leq B_{n+1}$ and $A_{n+1}< {n+1}$).
 The other one splits $A_n-1$ times, then all the lifts lifts at level $n+A_n$ strongly grow forever. \\
 \indent {\rm  3)} If $B_n \geq n$ and $A_n \geq n$, then all lifts strongly split at least $n-1$
 times.
\end{prop}

\begin{prop}\label{weaklysplits}
Let $\sigma$ be a weakly splitting cycle of $f_n$ ($n\geq 2$). Then
one lift behaves the same as
 $\sigma$ and the other one weakly grows and then strongly splits.
\end{prop}

\bigskip

To prove these propositions, we need  the following lemmas.

\begin{lem}\label{anbn-2}
Let $\sigma$ be a growing cycle of $f_n$ ($n\geq 2$). Then
\begin{eqnarray}\label{an-2}
a_{n+1}(x_i) \equiv 1 \quad ({\rm mod} \ 4),
\end{eqnarray}
\begin{eqnarray}\label{bn-2}
2b_{n+1}(x_i+p^n t) \equiv  b_n(x_i) (1+a_n(x_i)) \quad ({\rm mod} \
4).
\end{eqnarray}
\end{lem}
\begin{proof} Taking $p=2$ and $r=2$ in (\ref{an}),  we get
\[a_{n+1}(x_i) \equiv a_n^2(x_i)  \quad ({\rm mod} \ 2^n). \]
Since $n\geq 2$ and $a_n \equiv 1 \ ({\rm mod} \ 2)$, we obtain
(\ref{an-2}).

Taking $p=2$ and $r=2$ in (\ref{bn}), we get
\[2b_{n+1}(x_i+2^n t) \equiv t(a_n(x_i)^2-1) + b_n(x_i) (1+a_n(x_i)) \quad ({\rm mod} \ 2^n).\]
Since $n\geq 2$ and $a_n(x_i) \equiv 1 \ ({\rm mod} \ 2)$, we obtain
(\ref{bn-2}). \end{proof}

\begin{lem}\label{ana_n+1}
Let $\sigma$ be a splitting cycle of $f_n$.
 If $A_n <n$, then $A_{n+1}=A_n$ and if $A_n \geq n$, then $A_{n+1}\geq
 n$. Consequently, \begin{eqnarray}\label{A_nA_n+1}
\min \{A_{n+1},n\}=\min \{A_n, n\}.
 \end{eqnarray}
\end{lem}

\begin{proof} We need only to notice that we have $a_{n+1}\equiv a_n \ ({\rm
mod} \ 2^n)$ when $\sigma$ splits. \end{proof}

\begin{lem}\label{new-anbn-2}
Let $\sigma=(x_1,\dots, x_k)$ be a splitting cycle of $f_n$. Then
for $1\leq i \leq k$ and for $ t=0$ or $1$, we have
\begin{eqnarray}\label{b_n+1}
2 b_{n+1}(x_i+2^n t) \equiv b_{n}(x_i) + t (a_n(x_i)-1)  \quad ({\rm
mod} \ 2^{n}).
\end{eqnarray}
Consequently, we have
\begin{eqnarray}\label{B_n+1}
  B_{n+1}(x_i+2^nt)= B_{n}(x_i)-1 \qquad {\rm{if}} \ B_{n}(x_i)<
  \min \{A_n(x_i), \ n\},
\end{eqnarray}
\end{lem}
\begin{proof} Since $\sigma $ splits, taking $p=2$ and $r=1$ in (\ref{bn}),
we obtain the result. \end{proof}


\medskip

The following lemma concerns an elementary property of polynomials
on $\mathbb{Z}_2 $.
\begin{lem}\label{polyn}
Let $h \in \mathbb{Z}_2[x]$.  If $a\equiv b \ ({\rm mod} \ 2) $,
then $h'(a)\equiv h'(b) \ ({\rm mod} \ 4)$. Furthermore, if $h'(a)
\equiv 1 \ ({\rm mod} \ 2)  $, then $h'(a)h'(b) \equiv 1 \ ({\rm
mod} \ 4)$ .
\end{lem}

\begin{proof}
It suffices to notice that the coefficient of $x^{2k+1}$ in $h'(x) $
is equal to $0 \ ({\rm mod} \ 2)$.
\end{proof}
\medskip

\begin{lem}\label{after-growth}
Let $\sigma $ be a growing $k$-cycle of $f_n$ ($n\geq 1$). Then its
lift strongly grows or strongly splits.
\end{lem}

\begin{proof} Let $x_1$ be a point in $\sigma $. What we have to show is
$a_{n+1}(x_1) \equiv 1 \ ({\rm mod} \ 4) $. Since $\sigma$ is a
growing $k$-cycle, we have
$$f^{k}(x_1)\equiv x_1 \ ({\rm mod} \ 2^n)
,\qquad a_n(x_1) = (f^{k})'(x_1) \equiv 1 \ ({\rm mod} \ 2). $$ So,
by Lemma \ref{polyn}, we have $$ a_{n+1}(x_1)=
(f^{2k})'(x_1)=(f^{k})'(x_1)(f^{k})'(f^{k}(x_1)) \equiv 1 \ ({\rm
mod} \ 4). $$
\end{proof}

A direct consequence is the following result.
\begin{cor}\label{grow-two-times}
If a cycle grows twice (maybe between the two growths, there are
several splittings), then all the lifts will grow forever.
\end{cor}

\begin{proof}
 Let $\tilde{\sigma}$ be the lift of the growing cycle $\sigma$.
 Assume
 that after several times of splitting, one of lifts of $\tilde{\sigma} $ grows (then all the lifts at the same level grow).
 By Lemma \ref{after-growth}, this growing lift at a level $n \ge 2$ must strongly grow.
 Thus by Proposition \ref{grows}, the lifts will grow
 forever.
\end{proof}

\medskip
We are now going to prove Propositions
\ref{grows}-\ref{weaklysplits}.
\medskip

 \textit{Proof of Proposition
\ref{grows}.} If $\sigma$ grows, then by (\ref{an-2}), the lift of
$\sigma$ strongly grows or strongly splits. If $\sigma$ strongly
grows,
 then by
(\ref{bn-2}), we have
\[2b_{n+1}(x_i+p^n t) \equiv  2b_n(x_i) \quad ({\rm mod} \ 4).\]
Thus \[b_{n+1}(x_i+p^n t) \equiv b_n(x_i) \not\equiv 0 \quad ({\rm
mod} \ 2).\] Hence the lift of $\sigma$ strongly grows.

If $\sigma$ weakly grows,
 then by
(\ref{bn-2}), we have
\[2b_{n+1}(x_i+p^n t) \equiv  0 \quad ({\rm mod} \ 4).\]
Thus \[b_{n+1}(x_i+p^n t) \equiv 0 \quad ({\rm mod} \ 2).\] Hence
the lift of $\sigma$ strongly splits. \hfill $\Box$

\medskip
\textit{Proof of Proposition  \ref{stronglysplits}.} First notice
that if $\sigma$ strongly splits then $a_n \equiv 1 \ ({\rm mod} \
4)$. Since $n\ge 2$, by Lemma \ref{lem1} we have $a_{\ell} \equiv 1
\ ({\rm mod} \ 4)$ for all $\ell
>n$. So, all the lifts
strongly grow or strongly split.

Proposition \ref{stronglysplits} contains three cases which are
defined by some conditions on $A_n$ and $B_n$. If such a condition
is satisfied, we say  $\sigma $ or $(A_n,B_n) $ belongs to  the
corresponding case.
\smallskip

{\em Case 1: $\min \{A_n, n\} > B_n $}. Recall that by Lemma
\ref{lem1}, both $\min \{A_n(x_i), n \}$ and $\min \{A_n(x_i),
B_n(x_i), n \}$ are independent of $x_i$. Thus in this case, we can
simply write $A_n$ and $B_n$. By (\ref{B_n+1}), we have
$B_{n+1}=B_n-1$.  Thus by (\ref{A_nA_n+1})
\[\min \{A_{n+1}, n+1\} \geq  \min \{A_{n+1}, n\} = \min \{A_n, n\} > B_n >B_{n+1} .\]
 Hence the lifts of $\sigma$ still belong to Case 1.
By induction, we know that after $\ell:=B_n$ times, $B_{n+\ell}=0$
(i.e. $b_{n+\ell} \not =0$ mod $p$). Since $\sigma$ strongly splits,
we have $a_{n+\ell} \equiv 1 \ ({\rm mod} \ 4)$. Thus the lifts at
level $n+\ell$ strongly grow. That is to say all lifts of $\sigma$
split $B_n-1$ times, then all the lifts lifts at level $n+B_n$
strongly grow forever.

\smallskip

{\em Case 2:  $A_n\leq B_n$ and $A_n < n$}. Since $A_n(x_i) <n $ for
some $i$, implies for all $1\leq i \leq k$, $A_n(x_i)=A_n(x_j)$. We
can also deduce that if $A_n(x_i)\leq B_n(x_i)$ for some $i$ then
for all $i$ $A_n(x_i)\leq B_n(x_i)$. Otherwise, if $A_n(x_j)>
B_n(x_j)$ for some $j$, then by Lemma \ref{lem1}, $B_n(x_i)=B_n(x_j)
< A_n(x_j)=A_n(x_i)$ which leads to a contradiction. So in this
case, we can choose any $1\leq i \leq k$ and we simply write $A_n$
and $B_n$. By Lemma \ref{ana_n+1}, we have $A_{n+1}=A_n$. Since
$B_n\geq A_n$, there exists one $t$ such that
\[ b_n + t(a_n-1) \equiv 0 \quad ({\rm mod} \ 2^{A_n+1}),\]
and the other one which we can write as $1-t$ such that
\[ b_n + t(a_n-1) \not\equiv 0 \quad ({\rm mod} \ 2^{A_n+1}).\]
Hence by (\ref{b_n+1}), for one lift of $\sigma$ $B_{n+1} \geq A_n$
and for the other one $B_{n+1} = A_n-1$. Thus for one lift,
$A_{n+1}=A_n\leq B_{n+1}$, and $A_{n+1}=A_n < n+1$. Therefore, this
lift belongs to Case 2. For the other one, $B_{n+1} =
A_n-1=A_{n+1}-1< A_{n+1}$, and $B_{n+1} = A_n-1<n+1$. Thus this lift
belongs to Case 1.
 By induction, we know that one lift of $\sigma$ behaves the same as $\sigma$
 (i.e., strongly splits and satisfies the condition of Case 2 at level $n+1$) and the other one
 splits $A_n-1$ times, then the lifts strongly grow.

 \smallskip

{\em Case 3:  $B_n \geq n$ and $A_n \geq n$}. First we notice that
by Lemma \ref{lem1}, if for some $1\leq i \leq k$, we have $B_n(x_i)
\geq n$ and $A_n(x_i) \geq n$, then for all $1\leq i \leq k$, the
same property established. The following statement will be the same
if we choose another $i$. So we still simply write $A_n$ and $B_n$.
By the definition of $b_n$, if the cycle splits, the order of $b_n$
deceases at most one when the level goes up one step. Since $B_n
\geq n$, we have $B_{n+1} \geq n-1$, and if $n\geq 2$, the lifts of
$\sigma$ still strongly split. Thus by induction, the lifts of
$\sigma$ split at least $n-2$ times. But after that we can not give
any more information. \hfill $\Box$

\medskip
\textit{Proof of Proposition \ref{weaklysplits}.} Since $\sigma $
weakly splits, $a_{n+1} \equiv a_n \equiv 3 \ ({\rm mod} \ 4)$. Thus
$A_{n+1}=A_n=1 < n$ and $B_n \geq 1= A_n $. Thus $(A_n, B_n) $
belongs to Case 2 in Proposition \ref{stronglysplits}. By the proof
of Proposition \ref{stronglysplits}, we know that for one lift of
$\sigma $, $B_{n+1} \geq A_n$ and then $A_{n+1}=A_n \leq B_{n+1} $.
Thus this lift behaves the same as $\sigma $. For the other lift,
$B_{n+1}= A_n-1=0 $. Hence this second lift weakly grows, and then
its lift strongly splits by Proposition \ref{grows}. Therefore, we
complete the proof.
 \hfill $\Box$

\medskip

\setcounter{equation}{0}

\section{Minimal decomposition}\label{formation}
 If a cycle
always grows (grows forever) then it will produce a minimal
component of $f$. If a cycle always splits (splits infinite times)
then it will produce a periodic orbit of $f$. If a cycle grows
tails, it will produce  an attracting periodic orbit with an
attracting basin. We shall describe this more precisely.

Let $\sigma=(x_1,\dots,x_k)$ be a cycle of $f_n$. Recall that in
this case $\sigma $ is called a $k$-cycle at level $n$. Let
\[
  \mathbb{X}:=\bigsqcup_{i=1}^k  (x_i+p^n\mathbb{Z}_p).
\]
There are  four special situations for the dynamical system $f :
\mathbb{X} \to \mathbb{X}$.\\

\indent (S1) Suppose $\sigma $ grows tails. Then $f$ admits a
$k$-periodic orbit with one periodic point in each ball
$x_i+p^n\mathbb{Z}_p \ (1\leq i \leq k)$, and all other points in
$\mathbb{X}$ are attracted into this orbit. In this situation, if
$x$ is a point in the $k$-periodic orbit, then $|(f^k)'(x)|_p<1 $
since $(f^k)'(x)=a_m(x) \equiv 0 \ ({\rm mod} \ p^{m}) $ for all
$m\geq n $. The periodic orbit $(x, f(x), \cdots, f^{k-1}(x))$ is
then attractive.\\

\indent (S2) Suppose $\sigma $ grows and its lifts always grow. Then
$f$ is transitive (minimal) on each $\mathbb{X}/p^m\mathbb{Z}_p \
m\geq n$. Thus, by Theorem \ref{minimal-part-to-whole},  $f$ is
minimal on $\mathbb{X}$. In this case, we say that $\sigma $ is a
starting
growing cycle at level $n$.\\

\indent (S3) Suppose $\sigma$ splits and there is a splitting lift
at each level larger than $n$. Then there is a $k$-periodic orbit
with one periodic point in each $x_i+p^n\mathbb{Z}_p \ (1\leq i \leq
k)$. We say that $\sigma $ is a starting splitting cycle at level
$n$. In this situation, if $x$ is a point in the $k$-periodic orbit,
then $(f^k)'(x)=1 $ since $(f^k)'(x)=a_m(x) \equiv 1 \ ({\rm mod} \
p^{m}) $ for all $m\geq n $. Thus the periodic orbit $(x, f(x),
\cdots, f^{k-1}(x))$ is indifferent.\\

\indent (S4) Suppose  $\sigma=(x_1,\dots,x_k) $ partially splits
($p\geq 3$).  Then by Proposition \ref{partiallysplits}, there is
one lift of length $k$ which still partially splits like $\sigma $.
Thus there is a $k$-periodic orbit with one periodic point in each
$x_i+p^n\mathbb{Z}_p \ (1\leq i \leq k)$. In this situation, if $x$
is a point in the $k$-periodic orbit formed above, then
$|(f^k)'(x)|_p=1 $ since $(f^k)'(x)=a_m(x) \not\equiv 0,1 \  ({\rm
mod} \ p^{m}) $ for all $m\geq n $. Hence, the periodic orbit $(x,
f(x), \cdots, f^{k-1}(x))$ is  indifferent.\\



Now we can deduce  all possible periods of the polynomial systems on
$\mathbb{Z}_p$.
\begin{thm}
\label{possible-periods} Let $f \in \mathbb{Z}_p$ with
$\deg f \ge
2$.\\
\indent \mbox{\rm 1)} \  If $p\ge 5$, the lengths of periodic orbits
are of the form $a b$
with $a| (p-1)$ and $1\le b \le p$;\\
\indent \mbox{\rm 2)} \  If $p=3$, the lengths of periodic orbits
must be
$1,2,3,4,6$ or $9$;\\
\indent \mbox{\rm 3)} \  If $p=2$, the lengths of periodic orbits
must be $1,2$ or $4$.\\
\indent \mbox{\rm 4)} \ Let $p=2$. If there is $4$-periodic orbit,
then $f_1$ should be a permutation on $\mathbb{Z}/2\mathbb{Z}$.
There is no $4$-periodic orbit for quadratic polynomials.

\end{thm}

\begin{rem}
The statements 1)-3) were due to Pezda (\cite{Pezda}). By using the
idea we explained in the preceding sections which is quite different
from that of Pezda, Desjardins and Zieve (\cite{DZ}) also gave 1)
and 2). The statement 4) is new.
\end{rem}

\begin{proof} We only show 3) and 4), because the proofs of 1) and 2) are
similar and can be found in \cite{DZ} and \cite{Pezda}.

Notice that any periodic orbit comes from an infinite sequence of
splitting of some cycle, and that the length of the periodic orbit
is the length of the starting splitting cycle. So, what we want to
study are all possible lengths of  starting splitting cycles.

The possible lengths of cycles at the first level (i.e. the cycles
of $f_1$ on $\mathbb{Z}/2\mathbb{Z}$) are $1$ and $2$. Notice that
the growth of  length must be multiplied $2$, according to our
discussion  in the preceding sections. So, the possible lengths of
cycles are $2^k$ ($k\geq 0$). However, by Corollary
\ref{grow-two-times}, if a cycle grows twice it will grow forever.
There,  any cycle of length $2^k$ ($k\geq 3$), which must have grown
twice, can not be a starting splitting cycle. Hence the lengths of
starting splitting cycles can only be $1,2,4$. This completes the
proof of 3).

If there is a periodic orbit of length $4$, there must be a starting
splitting cycle of length $4$. This is possible only in  the
following case: at the first level, $f_1 $ admit a $2$-cycle.
Otherwise it needs to grow twice and then its lifts will grow
forever.  This will produce a clopen minimal set not a periodic
orbit. This is the first part of 4). For the second part of 4), one
can see from our study on the quadratic polynomials in Section
\ref{2-adic} (Theorems \ref{x^2-lambda}-\ref{x^2+x-d-3}).
\end{proof}

\medskip

\begin{thm}\label{thm-decomposition}Let $f \in \mathbb{Z}_p[x]$ with $\deg
f \ge 2$. We have the following decomposition
$$
     \mathbb{Z}_p = A \bigsqcup B \bigsqcup C
$$
where $A$ is the finite set consisting of all periodic points of
$f$, $B= \bigsqcup_i B_i$ is the union of all (at most countably
many) clopen invariant sets such that each $B_i$ is a finite union
of balls and each subsystem $f: B_i \to B_i$ is minimal, and each
point in $C$ lies in the attracting basin of a periodic orbit or of
a minimal subsystem.
\end{thm}

\begin{proof} We first explain that there are only finitely many periodic
points. In fact, by Theorem \ref{possible-periods}, there are only
finitely many   possible lengths of periods. Periodic points are
solutions of the equations $f^{q_i}(x)=x$ with $\{q_i\} $ being one
of possible lengths of periods. Since $\deg f\ge 2$, each equation
admits a finite number of solutions. So, there is only a finite
number of periodic points.

We start from the second level. Decompose $\mathbb{Z}_p$ into $p^2$
balls with radius $p^{-2}$. Each ball is identified with a point in
$\mathbb{Z}/p^2 \mathbb{Z}$. The induced map $f_2$ admits some
cycles. The points outside any cycle are mapped into the cycles. The
ball corresponding to such a point  will be put into the third part
$C$. From now on, we really start our analysis with cycles at level
$n\ge 2$. Let  $\sigma=(x_1,\dots,x_k)$ be a cycle at level $n \geq
2$. Let
$$
   \mathbb{X}=\bigsqcup_{i=1}^k
(x_i+p^n\mathbb{Z}_p).
$$

{\sc Suppose
 $p\geq 3$}. We distinguish four cases.\\
 \indent (P1) {\em $\sigma$ grows tails.}  Then by (S1), the clopen set $\mathbb{X}
$ consists of a $k$-periodic orbit and other points are attracted by
this periodic orbit. So, $\mathbb{X} $ contributes to the first part
$A$ and the third part $C$.

\indent (P2) {\em $\sigma $ grows}. Then by Proposition
\ref{cycle-p>3}, $\sigma $ is in the situation (S2). Therefore
$\mathbb{X}$ is a minimal component. So, $\mathbb{X}\subset B$.

\indent (P3) {\em  $\sigma$ splits}.  Then we shall apply
Proposition \ref{splits-p3}.
\begin{itemize}
  \item If $\sigma$ belongs to Case 1 described by Proposition
\ref{splits-p3}, then after finitely many times of splitting, the
lifts will grow forever and so they are in the situation of (S2).
Therefore we get a finite number of minimal components, all
belonging to $B$.
\item If $\sigma$ belongs to Case 2, then there is one
lift of $\sigma$ sharing the property (S3), and other lifts
different from the cycle containing the periodic orbit (at any level
$m \geq n+1 $)  find themselves  in the situation (S2) after
finitely many times of lifting. Therefore, we get a  periodic orbit
and countable infinite minimal components.
\item If $\sigma$ belongs to
Case 3, then $\sigma $ splits into $p^{n} $ cycles at level $2n$.
These cycles at level $2n$ may continue this procedure of analysis
of (P3). But this procedure can not continue infinitely, because
there is only a finite number of periodic points. So, all these
cycle may continue to split but they must end with their lifts
belonging either to Case 1 or Case 2 in Proposition \ref{splits-p3}.
So, $\mathbb{X}$ contributes to both $A$ and $B$.
\end{itemize}

\indent (P4) {\em  $\sigma$ partially splits}. Then $\sigma$ is in
the situation (S4).  Thus there comes out a periodic orbit. Suppose
$\sigma_m $ is the lift of $\sigma$ containing the periodic orbit at
level $m \geq n+1 $. If $\sigma_m $ belongs to Case 1 in Proposition
\ref{partiallysplits}, then the other lifts different from
$\sigma_{m+1}$, will be in the situation (S1) after finite times. If
$\sigma_m $ belongs to Case 2 in Proposition \ref{partiallysplits},
then each of other lifts different from $\sigma_{m+1}$, split to be
$p^{nd-1}$ cycles at level $nd$. We then go to (P3) for these cycles
at level $nd$.

\smallskip
{\sc Suppose
 $p= 2$}. We distinguish five cases.\\
 \indent (Q1) {\em $\sigma$ grows tails}. Then $\sigma$ is in
the situation (S1). We have the same conclusion as (P1) above.

\indent (Q2) {\em $\sigma$ strongly grows}.  Then by Proposition
\ref{grows}, $\sigma $ is  in the situation (S2).  We have the same
conclusion as (P2) above.

\indent (Q3) {\em $\sigma$ strongly splits}.  By Proposition
\ref{stronglysplits}, the arguments are  the same as (P3): The
procedures will be ended if the condition  1) or 2) in Proposition
\ref{stronglysplits} is satisfied. If the condition 3) in
Proposition \ref{stronglysplits} is satisfied, we repeat the
analysis of (Q3) for the lifts of $\sigma$. But the procedures will
be eventually ended with the condition  1) or 2), because there is
only a finite number of periodic points.

\indent (Q4) {\em $\sigma$ weakly grows}. Then by Proposition
\ref{grows}, the lift of $\sigma $ strongly splits.  We are then in
the case (Q3).

\indent (Q5) {\em $\sigma$ weakly splits}.  By Proposition
\ref{weaklysplits}, then one lift is in the situation (S3) which
produces a periodic orbit, and the other lifts different from the
cycle containing the periodic orbit, at any level $m \geq n+1 $,
will weakly grow.  Then we are in the case (Q4).

All  the above procedures will stop. So, we get the decomposition in
finite steps.
\end{proof}

  We have excluded the affine polynomials from the theorem. Exactly speaking, the conclusion
  is false for   affine polynomials. For example, every points in $\mathbb{Z}_p $ are fixed
  by
  $f(x):=x$.
Anyway, affine polynomials have been fully studied in \cite{FLYZ}.

\begin{cor}\label{thm-neutral-fixed-points}
Let  $f \in \mathbb{Z}_p[x] $ with $\deg f \ge 2$.  If $f$ admits an
 indifferent fixed point or a periodic orbit, then there exists
a sequence of minimal components with their diameters and their
distances from the fixed point or the periodic orbit tending to
zero.
\end{cor}
\begin{proof} Suppose $(x_1,\dots,x_k)$ is an indifferent periodic orbit. Let
$x^{(n)}_j \in \mathbb{Z}/p^n \mathbb{Z} $ and $x^{(n)}_j \equiv x_j
\ ({\rm mod} \ p^{n})$ for $1\leq j \leq k $. Then
$\sigma_n=(x^{(n)}_1,\dots,x^{(n)}_k) $ is a splitting or partially
splitting cycle at level $n$. By the procedures of the
decomposition, the cycle $\sigma_n$ should be in the situation (S3)
or (S4). That is to say $\sigma_n$ splits for all $n$ or $\sigma_n$
partially splits for all $n$.

 Since there are only finite number of periodic orbits,
 for any $\epsilon>0$ small enough, there is no other periodic orbits in the $\epsilon$
 neighborhood of the orbit $(x_1,\dots,x_k)$.
 Take $n$ such that $p^{-n}< \epsilon$. Then the lifts of $\sigma_n $ which are
different to $\sigma_{n+1} $ will never split infinitely. Hence they
will grow after finite times. Then all the lifts of $\sigma_n $
which are different to $\sigma_{n+1} $, considered as union of
balls, consist of finite number of minimal components. Since these
balls are contained in $x_j+p^{n}\mathbb{Z}_p$ for each $j$
respectively. Thus there is a minimal component such that the
diameter and the distance to the orbit $(x_1,\dots,x_k)$ are all
less than $p^{-n}$. The result is obtained if we consider infinitely
$n$ and find one minimal component for each $n$.
\end{proof}

\setcounter{equation}{0}

\section{Conjugacy classes of Minimal subsystems}\label{typicalversion}
Recently, Chabert, Fan and Fares \cite{CFF} proved that minimal sets
of a $1$-Lipschitz map are Legendre sets. We shall prove that
minimal sets of a polynomial are some special Legendre sets. A set
$E \subset \mathbb{Z}_p$ is a {\em Legendre set} if for any $s\ge 1$
and any $x \in E/p^s \mathbb{Z}_p$, the number
$$
      q_s := \mbox {\rm Card  }\left\{y\in E/p^{s+1}\mathbb{Z}_p:  \ y \equiv x \!\mod p^s\right\}
$$
is independent of $x \in E/p^s \mathbb{Z}_p$. Let
     $$ p_s := q_1q_2\cdots q_s \qquad (\forall s\ge 1).$$
It is clear that $p_s = \mbox {\rm Card  } E/p^s \mathbb{Z}_p$. We
call $(p_s)_{s\ge 1}$ the {\em structure sequence} of $E$.  Consider
the inverse limit
$$
      \mathbb{Z}_{(p_s)} : = \lim_{\leftarrow} \mathbb{Z}/p_s
      \mathbb{Z}.
$$
This is a profinite group, usually called an {\em odometer}, and the
map $\tau: x \mapsto x+1$ is called the {\em adding machine} on
$\mathbb{Z}_{(p_s)}$.

\begin{thm}[\cite{CFF}]
Let $E$ be a clopen set in $\mathbb{Z}_p$ and $f: E \rightarrow E$
be a $1$-Lipschitz map. If the  dynamical system $(E, f)$ is
minimal, then $f$ is an isometry, $E$ is a Legendre set and the
system $(E, f)$ is  conjugate to the adding machine
$(\mathbb{Z}_{(p_s)}, \tau) $ where $(p_s)$ is the structure
sequence of $E$. On the other hand, on any Legendre set there exists
at least one minimal map.
\end{thm}

We improve the above result in the case of polynomials by giving
more information on the structure sequence.

\begin{thm}\label{structure-minimal} Let $f \in \mathbb{Z}_p[x]$ with $\deg
f \ge 2$. If $E$ is a minimal clopen invariant set of $f$, then $f :
E \to E$ is conjugate to the adding machine on an odometer
$\mathbb{Z}_{(p_s)}$, where  $$(p_s) = (k, kd, k dp, kdp^2,
\cdots)$$ with some $1\leq k \in \mathbb{N}, k\leq p$ and $d|
(p-1)$.
\end{thm}

\begin{proof}
By our previous discussion on the cycles of $f_n$ on
$\mathbb{Z}/p^n\mathbb{Z}$, a clopen minimal set $E$ is formed when
a cycle grows forever. If $n$ is the starting level for the  cycle
to grow, then $E$ is a union of some balls with radius $p^{-n}$.
Therefore, for $s\geq n $, every nonempty intersection of $E$ with a
ball of radius $p^{-s}$ contains $p$ balls of radius $p^{-(s+1)}$.
That is to say $q_s =p$.  From the cycle at the first level to the
starting growing cycle at level $n$, the growth of cycle length  is
multiplied by $1$, $p$ or some  $d$ satisfying $d|(p-1)$. That is to
say for $1 \leq s< n$, every nonempty intersection of $E$ with a
ball of radius $p^{-s}$ contains the same number ($1, p$ or $d$) of
balls of radius $p^{-(s+1)} $. Thus $E$ is a Legendre set. To
determine $p_s$ for $1\le s<n$, we distinguish three cases: $p\ge
5$, $p=3$, $p=2$.

{\em Case $p\geq 5$}.  In this case, when a cycle grows, its lift
grows forever. A cycle at level $1$ may start with growing, several
times of splitting or several times of partially splitting and then
the lifts grow forever. Therefore, there are three ways to form a
minimal set. We show the three ways by the growth of cycle length as
follows ($k$ being the length of the cycle $\sigma$ at the
level $1$).\\
\indent Case 1. $\sigma$ grows:
\begin{eqnarray*}
(k,kp,kp^2,\dots),
\end{eqnarray*}
\indent Case 2. $\sigma$ splits:
\begin{eqnarray*}
(k,k,\dots, k,kp,kp^2,\dots),
\end{eqnarray*}
\indent Case 3. $\sigma$ partially splits:
\begin{eqnarray*}
(k,kd,\dots, kd,kdp,kdp^2,\dots), \quad  d|(p-1), d\geq 2.
\end{eqnarray*}
The above three cases correspond to three kinds of adding machines.
However, by the result of Buescu and Stewart \cite{BuescuStewart},
 the adding machines in both Case 1 and Case 2 are  conjugate to
$(\mathbb{Z}_{(p_s)}, \tau) $ where $p_s=(k,kp,kp^2,\dots)$. In Case
3, the adding machines are all conjugate to $(\mathbb{Z}_{(p_s)},
\tau) $ where $p_s=(k,kd,kdp,kdp^2,\dots)$ and $d|(p-1), d\geq 2$.

{\em Case $p=3$}. We distinguish four cases.\\
 \indent Case 1. $\sigma$ grows and its lift also grows:
\begin{eqnarray*}
(k,kp,kp^2,\dots),
\end{eqnarray*}
 \indent Case 2. $\sigma$ grows but its lift splits:
\begin{eqnarray*}
(k,kp,\dots, kp, kp^2,\dots),
\end{eqnarray*}
\indent Case 3. $\sigma$ splits:
\begin{eqnarray*}
(k,k,\dots, k,kp,kp^2,\dots),
\end{eqnarray*}
\indent Case 4. $\sigma$ partial splits:
\begin{eqnarray*}
(k,kd,\dots, kd,kdp,kdp^2,\dots), \quad  d|(p-1), d\geq 2.
\end{eqnarray*}
Then  $(E, f)$ is conjugate to $(\mathbb{Z}_{(p_s)}, \tau) $ where
$p_s=(k,kd,kdp,kdp^2,\dots)$ with $1\leq k \leq p $, and $d|(p-1) $.

{\em Case $p=2$}.  We distinguish twelve  cases.
$$
(1,\underbrace{1,1, \dots 1}_{{\rm strongly \ split}}, 2, 2^2,
2^3,\dots), $$
$$ (1,\underbrace{1}_{{\rm strongly \ grows}}, 2,
2^2, 2^3,\dots),
$$
\[
(1,\underbrace{1}_{{\rm weakly \ splits}}, \underbrace{1}_{{\rm
weakly \ grows}},\underbrace{2, \dots 2}_{{\rm strongly \ split}},
2^2, 2^3,\dots), \]
\[ (1,\underbrace{1}_{{\rm weakly \ grows}},
\underbrace{2, \dots 2}_{{\rm strongly \ split}}, 2^2, 2^3,\dots),
\]
$$
(1,\underbrace{2,\dots, 2}_{{\rm strongly \ split}}, 2^2,
2^3,\dots), $$
$$ (1,\underbrace{2}_{{\rm strongly \ grows}}, 2^2,
2^3,\dots),
$$
\[
(2,\underbrace{2,\dots, 2}_{{\rm strongly \ split}}, 2^2,
2^3,\dots),
\]
$$ (2,\underbrace{2}_{{\rm strongly \ grows}}, 2^2,
2^3,\dots),
$$
\[
(2,\underbrace{2}_{{\rm weakly \ splits}},\underbrace{2}_{{\rm
weakly \ grows}}, \underbrace{2^2,\dots,2^2}_{{\rm strongly \
split}}, 2^3,\dots),
\]
$$ (2,\underbrace{2}_{{\rm weakly \ grows}},
\underbrace{2^2,\dots,2^2}_{{\rm strongly \ split}}, 2^3,\dots),
$$
$$
(2,\underbrace{2^2,\dots, 2^2}_{{\rm strongly \ split}}, 2^3,\dots),
$$
$$
(2,\underbrace{2^2}_{{\rm strongly \ grows}}, 2^3,\dots).
$$
In any of these cases,  the system $(E, f)$ is conjugate to
$(\mathbb{Z}_2, x+1) $.
\end{proof}

\setcounter{equation}{0}

\section{$2$-adic Quadratic Polynomials}\label{2-adic}
In this section, we undertake a full investigation on the minimal
decomposition of $2$-adic quadratic polynomial systems on
$\mathbb{Z}_2$ of the form:
$$f(x):=ax^2+bx+c \ (a,b,c \in \mathbb{Z}_2, a\neq 0) .$$

As we shall see, the system $f(x) = ax^2 +b x +c$ is conjugate to
one of the following quadratic polynomials $$ x^2-\lambda, \quad
x^2+bx, \quad x^2+x-d$$ where $\lambda \in \mathbb{Z}_2$, $b\equiv
1\ ({\rm mod} \ 2)$ and $\sqrt{d} \not\in \mathbb{Z}_2$.

Let us state our results on the minimal decomposition of
$(\mathbb{Z}_p, f)$. The proofs are postponed at the end of  this
section. By the way, we shall discuss the behavior of $f$ on the
field $\mathbb{Q}_p$.

\smallskip
If $a \equiv 1 \ ({\rm mod} \ 2)  $, then $\lim_{n\to \infty}
|f^n(x)|= \infty $ for any $ x\not \in \mathbb{Z}_2$. An elementary
calculation shows that $ax^2 + b x +c$ on $\mathbb{Z}_2 $ is
conjugate to $x^2+bx+c $ on $\mathbb{Z}_2 $ through the conjugacy $x
\mapsto ax$. If $a \equiv 0 \ ({\rm mod} \ 2) $, then $\lim_{n\to
\infty} |f^n(x)|= \infty $ for any $ x\not \in
\frac{1}{a}\mathbb{Z}_2$ and $ax^2+bx+c $ on
$\frac{1}{a}\mathbb{Z}_2 $ is conjugate to $x^2+ax+ac $ on
$\mathbb{Z}_2 $ through the conjugacy $x \mapsto ax $. Thus without
loss of generality, we need only to consider the quadratic
polynomials of the form
$$
x^2+bx+c \ (b,c \in \mathbb{Z}_2).
$$
We distinguish two cases according to $b \equiv 0 \ ({\rm mod} \ 2)$
or $b \equiv 1 \ ({\rm mod} \ 2)$

If $b \equiv 0 \ ({\rm mod} \ 2)   $, $x^2+bx+c $ is conjugate to
$$x^2-\lambda $$
 with $\lambda= \frac{b^2-4c-2b}{4} $, through the conjugacy
$x\mapsto x+\frac{B}{2}$.

\begin{thm}\label{x^2-lambda}
Consider the polynomial $f(x)=x^2-\lambda$ on $\mathbb{Z}_2$.\\
 \indent {\rm 1)} If $\lambda \equiv 0 \ ({\rm mod} \ 4) $, then $f$
 admits
 two attracting fixed points, one in  $4\mathbb{Z}_2 $ with  $2\mathbb{Z}_2 $ as its attraction basin, and the other one in $1+
 4\mathbb{Z}_2$ with $1+2\mathbb{Z}_2$ as its attraction basin.\\
 \indent {\rm 2)} If $\lambda \equiv 1 \ ({\rm mod} \ 4) $, then the whole
 $\mathbb{Z}_2 $ is attracted into a periodic orbit of period $2$ with one orbit point in  $4\mathbb{Z}_2 $
 and the other one in $3+
 4\mathbb{Z}_2$.\\
 \indent {\rm 3)} If $\lambda \equiv 2 \ ({\rm mod} \ 4) $, then $f$
 admits
 two attracting fixed points, one in  $2+4\mathbb{Z}_2 $  $2\mathbb{Z}_2 $ as its attraction basin, and the other one
 in $3+
 4\mathbb{Z}_2$ with $1+2\mathbb{Z}_2 $ as its attraction basin.\\
 \indent {\rm 4)} If $\lambda \equiv 3 \ ({\rm mod} \ 4) $, then the whole
 $\mathbb{Z}_2 $ is attracted into a periodic orbit of period $2$ with one orbit point in  $1+4\mathbb{Z}_2 $
  and the other one in $2+
 4\mathbb{Z}_2$.
\end{thm}

\medskip

If $b \equiv 1 \ ({\rm mod} \ 2)  $, then $x^2+bx+c $ is conjugate
to $$ x^2+x-d $$ where $d = \frac{(b-1)^2-4c}{4}\in \mathbb{Z}_2$,
through $x\mapsto x+ \frac{b-1}{2} $.  It is clear that $x^2 +x -d$
admits fixed points if and if only $\sqrt{d}\in \mathbb{Z}_2$. Thus
we need to study the case $x^2+x-d $ with $\sqrt{d} \in \mathbb{Z}_2
$ and the case $x^2+x-d $ with $ d\in \mathbb{Z}_2$ but $\sqrt{d}
\not\in \mathbb{Z}_2 $.

If $\sqrt{d} \in \mathbb{Z}_2$ (i.e. $x^2+x-d $ has a fixed point),
then $x^2+x-d$ conjugates to $$x^2+bx $$
 with $b=1-2\sqrt{d} $, through $x\mapsto x + \sqrt{d}$.

 If $b=1$, the minimal decomposition of
$x^2+x $ is as follows.

\begin{thm}\label{x^2+x} Consider the polynomial $f(x) =
x^2+x$ on $\mathbb{Z}_2[x]$.  There is  one fixed point $0$. We have
 $f(1+2\mathbb{Z}_2) \subset 2\mathbb{Z}_2$ and we can decompose $2\mathbb{Z}_2$
 into
\[2\mathbb{Z}_2=\{0\}\bigsqcup
\left(\bigsqcup_{n\geq 2} 2^{n-1} +2^n\mathbb{Z}_2 \right). \] Each
$2^{n-1} +2^n\mathbb{Z}_2$ ($n\geq 2$) consists of $2^{n-2}$ pieces
of minimal components:
\[2^{n-1} +t2^n +2^{2n-2}\mathbb{Z}_2, \quad t=0, \dots , 2^{n-2}-1.\]
\end{thm}

\medskip

Denote $\mathbb{N}^*=\mathbb{N}\setminus \{0\}$. If $b\equiv 1 \
({\rm mod} \ 2)$ but $b\neq 1$, we distinguish four subcases:
\begin{itemize}
  \item $b=1-4m $, $m\in \mathbb{Z}_2\setminus \{0\} $;
  \item $b=-1-4m $, $m\in \mathbb{Z}_2 $ with $v_2(m) \in 1+2 \mathbb{N}$;
  \item $b=-1-4m $, $m\in \mathbb{Z}_2 $ with $v_2(m) \in 2 \mathbb{N}^*$;
  \item $b=-1-4m $, $m\in \mathbb{Z}_2 $ with $v_2(m)=0$.
\end{itemize}

If $f(x)=x^2+(-1-4m)x$ with $v_2(m)=0$, then $f$ is conjugate to
$g(x)=x^2+(-1-4(-m-1))x$ with $v_2(-m-1)=v_2(m+1)\geq 1$ through
$x\mapsto x-4m-2 $. Thus the last case is reduces to the second and
the third case. So we need only consider the first three cases.

Before the statement of the following results, we would like to give
some terminology to simplify our statements.

We say a $1$-cycle $(x)$ at level $n$ is of {\em type {\rm I-$[k]$}}
if it splits $k$ times then its lifts grow forever. In this case,
the ball $x+p^n\mathbb{Z}_p$ is decomposed into $p^k$ pieces of
minimal components. Such a component is a ball of radius $p^{-n-k}$.
Sometimes the ball $x+p^n\mathbb{Z}_p$ is said to be of type {\rm
I-$[k]$}.

We say the a $2$-cycle $(x,y)$  at level $n$ is of {\em type {\rm
II-$[k]$}} if it splits $k$ times then its lifts grow forever. In
this case, the union of two balls $(x+p^n\mathbb{Z}_p)\cup
(y+p^n\mathbb{Z}_p)$ is decomposed into $p^k$ pieces of minimal
components. Such a component is a union of two balls of radius
$p^{-n-k}$. The union $(x+p^n\mathbb{Z}_p)\cup (y+p^n\mathbb{Z}_p)$
is sometimes said to be of type {\rm II-$[k]$}. Remark that the
union $(x+p^n\mathbb{Z}_p)\cup (y+p^n\mathbb{Z}_p)$ may be a ball of
radius $p^{-n+1}$.

If an invariant subset $E \subset \mathbb{Z}_p$ is a union of
invariant subsets $F_n \subset \mathbb{Z}_p, \ n\in J\subset
\mathbb{N}$ where each  $F_n$ is of type I-$[k]$, we will denote
it as
\begin{eqnarray*}
  E=\bigsqcup_{n\in J} F_n -\{{\textrm{I-$[k]$}}\}.
\end{eqnarray*}
Similarly, if each $F_n \subset \mathbb{Z}_p, \ n\in J\subset
\mathbb{N}$ where each $F_n$ is a union of two balls of type
II-$[k]$, we will denote it as
\begin{eqnarray*}
  E=\bigsqcup_{n\in J} F_n -\{{\textrm{II-$[k]$}}\}.
\end{eqnarray*}

Now we are ready to state the following theorems.
\begin{thm}\label{x^2+(1-4m)x}
Consider $f(x)=x^2+(1-4m)x$ with $m\in \mathbb{Z}_2\setminus \{0\}$.
Then $f$ admits two fixed points $0$ and $4m $, and
$f(1+2\mathbb{Z}_2 ) \subset 2\mathbb{Z}_2 $. We can decompose
$2\mathbb{Z}_2$ as
\[
2\mathbb{Z}_2= \{0,4m\}\bigsqcup E_1
 \bigsqcup E_2 \bigsqcup E_3,
\]
where
\begin{align*}
  E_1&= \bigsqcup_{2\leq n < v_2(m)+3 } \left(2^{n-1}+2^n\mathbb{Z}_2\right) -\{{\textrm{I-$[n-2]$}}\},\\
  E_2&= \bigsqcup_{n > v_2(m)+3} \left(2^{n-1}+2^n\mathbb{Z}_2\right) -\{{\textrm{I-$[v_2(m)+1]$}}\},\\
  E_3&= \bigsqcup_{n > v(m)+3} \left(4m+2^{n-1}+2^n\mathbb{Z}_2\right) -\{{\textrm{I-$[v_2(m)+1]$}}\}.
\end{align*}
%
\end{thm}

\begin{thm}\label{x^2+(1-4m)x-v_2(m)-odd}
Consider $f(x)=x^2+(-1-4m)x$ with $v_2(m) \in 1+2\mathbb{N}$. Then
 $f$ admits two fixed points $0$ and
$4m+2$, and $f(1+2\mathbb{Z}_2) \subset 2\mathbb{Z}_2 $. We can
decompose $2\mathbb{Z}_2$ as
\[
2\mathbb{Z}_2= \{0,4m+2\}\bigsqcup E_1
 \bigsqcup E_2 \bigsqcup E_3,
\]
where
\begin{align*}
  E_1&= \bigsqcup_{n \geq 4} \left(4m+2+2^{n-2}+2^{n-1}\mathbb{Z}_2\right) -\{{\rm\text{II-$[1]$}}\},\\
  E_2&= \bigsqcup_{4\leq n \leq \lfloor {v_2(m)}/{2} \rfloor +3 }
  \left(2^{n-2}+2^{n-1}\mathbb{Z}_2\right) -\{{\textrm{II-$[2n-5]$}}\},\\
  E_3&= \bigsqcup_{ n > \lfloor {v_2(m)}/{2} \rfloor+3} \left(2^{n-2}+2^{n-1}\mathbb{Z}_2\right) -\{{\textrm{II-$[v_2(m)+1]$}}\}.
\end{align*}
\end{thm}

\begin{thm}\label{x^2+(1-4m)x-v_2(m)-even}
Consider $f(x)=x^2+(-1-4m)x$ with $v_2(m) \in 2\mathbb{N}^*$. Then
$f$ admits fixed points $0$ and $4m+2 $, and
$f(1+2\mathbb{Z}_2)\subset 2\mathbb{Z}_2 $. The invariant set
$2\mathbb{Z}_2$ admits the following form
\[
2\mathbb{Z}_2= \{0,4m+2\}\bigsqcup E_1
 \bigsqcup E_2 \bigsqcup E_3 \bigsqcup \left(2^{v_2(m)/2+1}+2^{v_2(m)/2+2}\mathbb{Z}_2 \right),
\]
where
\begin{align*}
  E_1&= \bigsqcup_{n \geq 4} \left(4m+2+2^{n-2}+2^{n-1}\mathbb{Z}_2\right) -\{{\rm\text{II-$[1]$}}\},\\
  E_2&= \bigsqcup_{4\leq n < {v_2(m)}/{2} +3 }
  \left(2^{n-2}+2^{n-1}\mathbb{Z}_2\right) -\{{\textrm{II-$[2n-5]$}}\},\\
  E_3&= \bigsqcup_{ n > {v_2(m)}/{2}+3} \left(2^{n-2}+2^{n-1}\mathbb{Z}_2\right) -\{{\textrm{II-$[v_2(m)+1]$}}\}.
\end{align*}
Denote $E= 2^{v_2(m)/2+1}+2^{v_2(m)/2+2}\mathbb{Z}_2$.
\begin{enumerate}
  \item If $v_2(m)=2 $ and $v_2(m-4)=3 $, then $E$ is of type {\rm
  II-$[4]$}.

\item If $v_2(m)=2 $ and $v_2(m-4)\geq 5 $, then $E$ is of type
{\rm{II-$[5]$}}.

\item If $v_2(m)=2 $ and $v_2(m-4)=4 $, then there exists a $2$-periodic orbit with one point
  $x_1\in 4+16\mathbb{Z}_2 $ and the other $x_2 \in 12+16\mathbb{Z}_2 $; and we can decompose $E$ as
$ E= \{ x_1, x_2\}\bigsqcup E_4$, where $$E_4=\bigsqcup_{k \geq 5}
\left( (x_1 + 2^{k-1}+ 2^k \mathbb{Z}_2)\cup (x_2 + 2^{k-1}+ 2^k
\mathbb{Z}_2)\right) -\{{\textrm{II-$[5]$}}\}. $$

\item If $v_2(m)\geq 4 $ and $v_2(m-2^{v_2(m)}) < v_2(m)+3 $, then
$E$ is of type {\rm{II-$[v_2(m-2^{v_2(m)})+1]$}}.

\item If $v_2(m)\geq 4 $ and $v_2(m-2^{v_2(m)}) \geq v_2(m)+3 $, then there exists  a $2$-periodic orbit with one point
  $x'_1\in 2^{{v_2(m)}/{2}+1}+2^{{v_2(m)}/{2}+3}\mathbb{Z}_2 $ and
  the other $x'_2 \in 2^{{v_2(m)}/{2}+1}+2^{{v_2(m)}/{2}+2}+2^{{v_2(m)}/{2}+3}\mathbb{Z}_2$;
  and we can decompose $E$ as
$ E= \{x'_1,x'_2\}\bigsqcup E'_4, $ where $$E'_4=\bigsqcup_{k \geq
v_2(m)/2+4} \left((x'_1 + 2^{k-1}+ 2^k \mathbb{Z}_2)\cup (x'_2 +
2^{k-1}+ 2^k \mathbb{Z}_2)\right) -\{{\textrm{II-$[v_2(m)+1]$}}\}.
$$
\end{enumerate}
\end{thm}

\medskip
Now we are left to study the polynomials $f(x)=x^2+x-d $ with $d\in
\mathbb{Z}_2 $ but $\sqrt{d} \not\in \mathbb{Z}_2 $.

We distinguish four cases.

\begin{thm}\label{x^2+x-d-0}
Consider $f(x)=x^2+x-d$ with $d\equiv 0 \ ({\rm mod} \ 4)$ and
$\sqrt{d}\not \in
 \mathbb{Z}_2$. Then $f(1+2\mathbb{Z}_2)\subset 2\mathbb{Z}_2$ and $2\mathbb{Z}_2$ is decomposed
 as finite number of minimal components.
Let $n_0=\lfloor {v_2(d)}/{2}\rfloor+1$.
  \begin{enumerate}
     \item If $v_2(d)=2 $ and $v_2(d-4)=3$,
   then $2\mathbb{Z}_2$ consists of three minimal components: $4\mathbb{Z}_2$, $2+8\mathbb{Z}_2$ and $6+8\mathbb{Z}_2$.

   \item If $v_2(d)=2 $ and $v_2(d-4)=4$, then $2\mathbb{Z}_2$ consists of five
               minimal components: $4\mathbb{Z}_2$,
               $2+16\mathbb{Z}_2$, $6+16\mathbb{Z}_2$,
               $10+16\mathbb{Z}_2$, and $14+16\mathbb{Z}_2$.

   \item If $v_2(d)\geq 3 $ and $v_2(d)$ is odd, then
   $
   2\mathbb{Z}_2=E_1\bigsqcup E_2,
   $
   where \begin{align*}
     E_1&=\bigsqcup_{2 \leq n \leq n_0} \left(2^{n-1}+2^n\mathbb{Z}_2\right) -\{{\textrm{I-$[n-2]$}}\}, \\
     E_2&=2^{n_0}\mathbb{Z}_2 -\{{\textrm{I-$[n_0-1]$}}\}.
   \end{align*}

   \item If $v_2(d)\geq 3 $ and $v_2(d)$ is even, then
   $
   2\mathbb{Z}_2=E'_1\bigsqcup E'_2 \bigsqcup E'_3,
   $
   where \begin{align*}
     &E'_1=\bigsqcup_{2 \leq n \leq n_0-1} \left(2^{n-1}+2^n\mathbb{Z}_2\right) -\{{\textrm{I-$[n-2]$}}\}, \\
     &E'_2=2^{n_0}\mathbb{Z}_2 -\{{\textrm{I-$[n_0-2]$}}\},\\
     &E'_3=2^{n_0-1}+2^{n_0} \mathbb{Z}_2 -\{{\textrm{I-$[v_2(d-2^{v_2(d)})-n_0]$}}\}.
   \end{align*}
  \end{enumerate}
\end{thm}

\begin{thm}\label{x^2+x-d-1}
Consider $f(x)=x^2+x-d$ with $d\equiv 1 \ ({\rm mod} \ 4)$ and
$\sqrt{d}\not\in
 \mathbb{Z}_2$. Then $f(2\mathbb{Z}_2) \subset 1+2\mathbb{Z}_2$ and $3+4\mathbb{Z}_2$ is of type
 {\rm II-$[1]$}. Let $d=5+8t $ with  $t\in \mathbb{Z}_2$. If $v_2(t) \leq 1$, then $1+4\mathbb{Z}_2$ is of type
 {\rm II-$[v_2(t)+2]$}. If $v_2(t) \geq 2$, then
            \[
            1+4\mathbb{Z}_2=\{x_1,x_2\} \bigsqcup E_1\bigsqcup
            E_2 \bigsqcup E_3,
            \]
with the form
 \begin{align*}
   E_1=(a+2^4\mathbb{Z}_2)\cup (f(a)+2^4\mathbb{Z}_2)
   -\{{\textrm{II-$[3]$}}\},\\
      E_2=(b+2^5\mathbb{Z}_2)\cup (f(b)+2^5\mathbb{Z}_2)
   -\{{\textrm{II-$[3]$}}\},\\
      E_3=\bigsqcup_{n\geq 6} (x_1+2^n\mathbb{Z}_2)\cup (x_2+2^n\mathbb{Z}_2)
   -\{{\textrm{II-$[3]$}}\},
 \end{align*}
 and $x_1,x_2$ is a $2$-periodic orbit such that $x_1 \in c+2^5\mathbb{Z}_2 $ and $x_2 \in f(c)+2^5\mathbb{Z}_2 $.
 Precisely,
   \begin{enumerate}
       \item If $v_2(t) = 2$ and $v_2(t-4)=3 $, then $a=1,b=25,c=9$.

\item If $v_2(t) = 2$ and $v_2(t-4) \geq 4 $, then $a=1,b=9,c=25$.

\item If $v_2(t) = 3$, then $a=9,b=1,c=17$.

\item If $v_2(t) \geq 4$, then $a=9,b=17,c=1$.
 \end{enumerate}
\end{thm}

\begin{thm}\label{x^2+x-d-2}
 Consider $f(x)=x^2+x-d$ with $d\equiv 2 \ ({\rm mod} \ 4)$ and $\sqrt{d}\not\in
 \mathbb{Z}_2$. Then $f(1+2\mathbb{Z}_2)\subset 2\mathbb{Z}_2 $.
 \begin{enumerate}
   \item If $v_2(d-2)=2$, then $2\mathbb{Z}_2$ is of type {\rm II-$[1]$}.
   \item If $v_2(d-2)=3$, then $8\mathbb{Z}_2 \cup
   (f(0)+8\mathbb{Z}_2)$ is of type {\rm II-$[1]$}, $(4+8\mathbb{Z}_2)\cup (f(4)+8\mathbb{Z}_2) $ consists
   of a $2$-periodic orbit with one point $x_1\in 4+8\mathbb{Z}_2$
   and the other $x_2 \in f(4)+8\mathbb{Z}_2 $, and for each $n\geq
   4$, $(x_1 + 2^{n}\mathbb{Z}_2)\cup (x_2+2^n \mathbb{Z}_2)$
   is of type {\rm II-$[2]$};
   \item If $v_2(d-2)\geq 4$, then $4+ 8\mathbb{Z}_2 \cup
   (f(4)+8\mathbb{Z}_2)$ is of type {\rm II-$[1]$}, $8\mathbb{Z}_2\cup (f(0)+8\mathbb{Z}_2) $ consists
   of a $2$-periodic orbit with one point $x_1\in 8\mathbb{Z}_2$
   and the other $x_2 \in f(0)+8\mathbb{Z}_2 $, and for each $n\geq
   4$, $(x_1 + 2^{n}\mathbb{Z}_2)\cup (x_2+2^n \mathbb{Z}_2)$
   is of type {\rm II-$[2]$}.
 \end{enumerate}
\end{thm}

\begin{thm}\label{x^2+x-d-3}
 For $f(x)=x^2+x-d$ with $d\equiv 3 \ ({\rm mod} \ 4)$, the ball $2\mathbb{Z}_2
 $ is mapped into the ball $1+2\mathbb{Z}_2$ which is the unique
 minimal component.
\end{thm}

We prove Theorems \ref{x^2-lambda}-\ref{x^2+(1-4m)x-v_2(m)-even}.
The proofs of Theorems \ref{x^2+x-d-0}-\ref{x^2+x-d-3} will be
omitted since they are similar to those of Theorems
\ref{x^2-lambda}-\ref{x^2+(1-4m)x-v_2(m)-even}.

\textit{Proof of Theorem \ref{x^2-lambda}.} Let $f(x)=x^2-\lambda$.
Then $f'(x)=2x $ and
$(f^2)'(x)=4x^3-4\lambda x $.\\
 \indent 1) If $\lambda \equiv  0 \ ({\rm mod} \ 4) $, then $2+4\mathbb{Z}_2
 $ and $3+4\mathbb{Z}_2 $ are mapped into $4\mathbb{Z}_2 $ and $1+4\mathbb{Z}_2
 $ respectively, and $4\mathbb{Z}_2 $ and $1+4\mathbb{Z}_2
 $ are mapped into themselves respectively. Consider the cycles $(0)
 $ and $(1)$ of $f_2$. We have
 \begin{eqnarray*}
a_2(0)=f'(0) \equiv  0 \ ({\rm mod} \ 2) \quad {\rm and} \quad
a_2(1)=f'(1) \equiv 0 \ ({\rm mod} \ 2).
 \end{eqnarray*}
Thus cycles $(0)$ and $(1)$ grow tails, hence there will form
 two attracting fixed points, one in  $4\mathbb{Z}_2 $ with basin $2\mathbb{Z}_2 $, and the other one in $1+
 4\mathbb{Z}_2$ with basin $1+2\mathbb{Z}_2 $.\\
\indent 2) If $\lambda \equiv  1 \ ({\rm mod} \ 4) $, then
$1+4\mathbb{Z}_2
 $ and $2+4\mathbb{Z}_2 $ are mapped into $4\mathbb{Z}_2 $ and $3+4\mathbb{Z}_2
 $ respectively, and $4\mathbb{Z}_2 $ and $3+4\mathbb{Z}_2
 $ are mapped into $3+4\mathbb{Z}_2 $ and $4\mathbb{Z}_2$ respectively. Consider the cycle $(0,3)$ of $f_2$. We have
 \begin{eqnarray*}
a_2(0)=(f^2)'(0) \equiv  0 \ ({\rm mod} \ 2) .
 \end{eqnarray*}
Thus cycle the cycle $(0,3)$ grows tails, hence there will form
 an attracting $2$-periodic orbit, with one periodic point in  $4\mathbb{Z}_2 $, and the other one in $3+
 4\mathbb{Z}_2$. We also see that the attracting basin is the whole $\mathbb{Z}_2
 $.

 The proofs of 3) and 4) are similar to the proofs of 1) and 2).
 \qed

\medskip

\textit{Proof of Theorem \ref{x^2+x}.} Let $f(x)=x^2+x$. We will use
a diagram to show the structure of the dynamics of $f$.

%


\begin{center}
\begin{pspicture}(0,0)(12,7)
\rput(7,3.5){\includegraphics{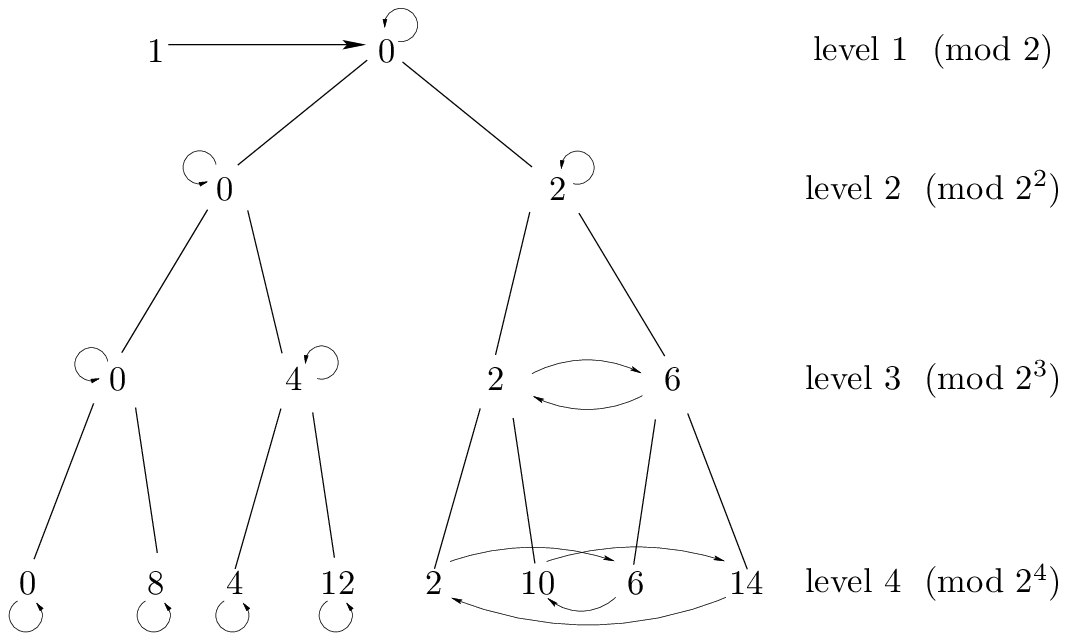}}

\end{pspicture}
\end{center}

%
%
%
%
%

At level $n$, the "$\rightarrow $" stands for the transformation of
the elements of $\mathbb{Z}/p^n\mathbb{Z}$ under $f_n$. Thus the
diagram shows that
\begin{eqnarray*}
f_1(1)=0, \ f_1(0)=0,\quad i.e., \ f(1+2\mathbb{Z}_2) \subset
2\mathbb{Z}_2 \ {\rm and} \ f(2\mathbb{Z}_2) \subset 2\mathbb{Z}_2.
\end{eqnarray*}
and
\begin{eqnarray*}
f_2(0)=0, \ f_2(2)=2,\quad i.e., \ f(4\mathbb{Z}_2) \subset
4\mathbb{Z}_2 \ {\rm and} \ f(2+4\mathbb{Z}_2) \subset
2+4\mathbb{Z}_2.
\end{eqnarray*}
Since $f(1+2\mathbb{Z}_2) \subset 2\mathbb{Z}_2  $ and
$f^{-1}(1+2\mathbb{Z}_2)= \emptyset$, we need only to consider
$2\mathbb{Z}_2$. From the diagram, we also see that $(0)$ is the
only cycle of $f_1$ with length $1$, and $(0),(2)$ are two lifts of
$(0)$.

 We will start our examination from the level $2$.
Since
\[ a_2(0)=f'(0)=1 \quad {\rm and} \quad b_2(0)= \frac{f(0)-0}{2^2}=0,\]
we have $a_2(0) \equiv 1 \ ({\rm mod} \ 4) $ and
\[ A_2(0)= \infty \quad {\rm and} \quad  B_2(0)=\infty.\]
Thus the cycle $(0)$ strongly splits.

Since
\[ a_2(2)=f'(2)=5 \quad {\rm and} \quad b_2(2)= \frac{f(2)-2}{2^2}=1,\]
we have $a_2(2) \equiv 1 \ ({\rm mod} \ 4) $ and
\[ A_2(2)=2 \quad {\rm and} \quad  B_2(0)=0.\]
Thus the cycle $(2)$ strongly grows which implies that the lift of
$(2)$ still grows, and so on. Hence $2+4\mathbb{Z}_2$ is a minimal
component.

By induction we know that for all $n\geq 2$
\[ A_n(0)= \infty \quad {\rm and} \quad  B_n(0)=\infty.\]
Thus the cycle $(0)$ of $f_{n-1}$ always splits to be two cycles
$(0)$ and $(2^{n-1})$ of $f_n$,
 and the number $0$ should be a fixed point.

Now for $n\geq 2$, let us consider the cycle $(2^{n-1})$ of $f_n$.
With the same calculations,
\[ a_n(2^{n-1})=2^{n}+1 \quad {\rm and} \quad b_n(2^{n-1})= 2^{n-2}.\]
Thus $a_n(2^{n-1})\equiv  1 \ ({\rm mod} \ 4)$ and
\[ A_n(2^{n-1})= n \quad {\rm and} \quad  B_n(2^{n-1})=n-2.\]
Hence, the cycle $(2^{n-1})$ strongly splits and $B_n < \min \{A_n,
n\} $. By Proposition \ref{stronglysplits}, the lift of $(2^{n-1})$
splits $B_n-1=n-3$ times then all lifts strongly grow. Thus there
are $2^{n-2}$ pieces of minimal components which constitute $2^{n-1}
+2^n\mathbb{Z}_2 $. They are
\[2^{n-1} +t2^n +2^{2n-2}\mathbb{Z}_2, \quad t=0, \dots , 2^{n-2}-1.\]
This concludes the proof of Theorem \ref{x^2+x}. \qed

\medskip
\textit{Proof of Theorem \ref{x^2+(1-4m)x}.} Let $f(x)=x^2+(1-4m)x$.
We see that there are two fixed points $0$ and $4m $, and
$1+2\mathbb{Z}_2 $ is mapped into $2\mathbb{Z}_2 $.  We are
concerned with the invariant subset $2\mathbb{Z}_2 $.

Consider
 $2^{n-1}+2^n\mathbb{Z}_2 \ (n\geq 2)$. We study the cycle
 $(2^{n-1})$ at level $n$. We have
 \[
   a_n(2^{n-1})=2^n-4m+1 \qquad b_n(2^{n-1})=2^{n-2}-2m,
 \]
thus $a_n(2^{n-1})\equiv  1 \ ({\rm mod} \ 4)$ and if $2\leq n <
v_2(m)+3 $,
\[
   A_n(2^{n-1})\geq n \qquad B_n(2^{n-1})=n-2.
 \]
If $n=2$, then the cycle $(2^{n-1})$ strongly grows. If $n>2 $,
 then the cycle $(2^{n-1})$ strongly splits and $B_n < \min \{A_n, n\} $.
 By Proposition \ref{stronglysplits}, the lift of $(2^{n-1})$ strongly splits $B_n-1=n-3$ times
then all lifts strongly grow. Thus we will obtain the part $E_1$ in
Theorem \ref{x^2+(1-4m)x}.

 If $ n > v_2(m)+3 $,
 \[
   A_n(2^{n-1})=v_2(m)+2 \qquad B_n(2^{n-1})=v_2(m)+1.
 \]
 Hence, the cycle $(2^{n-1})$ strongly splits and $B_n < \min \{A_n, n\} $. By Proposition \ref{stronglysplits},
  the lift of $(2^{n-1})$ strongly splits $B_n-1=v_2(m)$ times
then all lifts strongly grow. Hence we have the part $E_2$.

Consider $4m+2^{n-1}+2^n\mathbb{Z}_2$ ($ n > v_2(m)+3 $). Let $s_n
\equiv 4m + 2^{n-1} \ ({\rm mod} \ 2^n)$ and $0 \leq s_n < 2^{n} $.
We study the cycle $(s_n)$ at level $n$. We have
 \[
   a_n(s_n)=2s_n-4m+1 \qquad b_n(s_n)=\frac{s_n(s_n-4m)}{2^n},
 \]
thus $a_n(s_n)\equiv  1 \ ({\rm mod} \ 4)$ and
\[
   A_n(s_n)=v_2(m)+2 \qquad B_n(s_n)=v_2(m)+1.
 \]
 Hence, $B_n < \min \{A_n, n\} $. By Proposition \ref{stronglysplits}, the cycle $(s_n)$ strongly splits and
 the lift of $(s_n)$ strongly splits $B_n-1=v_2(m)$ times
then all lifts strongly grow. Therefore, we have the part $E_3$.
This completes the proof. \qed

\medskip
\textit{Proof of Theorem \ref{x^2+(1-4m)x-v_2(m)-odd}.} Let
$f(x)=x^2+(-1-4m)x$ with $v_2(m) \in 1+2\mathbb{N}$. We see that
there are two fixed points $0$ and $4m+2 $, and $1+2\mathbb{Z}_2 $
is mapped into $2\mathbb{Z}_2 $.

Since $0$ and $4m+2$ are two fixed points, there are cycles $(0) $
and $(t_n) $ at each level, where $t_n\equiv 4m +2 \ ({\rm mod} \
2^{n}) $ and $0 \leq t_n < 2^{n} $. Consider the cycles $(0)$ and
$(t_{n-2})$ at level $n-2$. By studying the $a_{n-2}, b_{n-2}$ of
these two cycles, we know that they weakly split. By Proposition
\ref{weaklysplits}, after splitting, half of lifts weakly grow. Thus
we will obtain two $2$-cycles: $(2^{n-2},2^{n-2}+2^{n-1})$ and
$(s_n, s_n+2^{n-1})$ at level $n$, where $s_n\equiv 4m +2+ 2^{n-2} \
({\rm mod} \ 2^{n-1})$ and $0 \leq s_n < 2^{n-1} $.

For each $ n \geq 4 $, we study the cycle $(s_n, s_n+2^{n-1})$ at
level $n$. We have
 \begin{eqnarray*}
   & &a_n(s_n)=8\left(4\left(\frac{s_n}{2}\right)^3-3(4m+1)\left(\frac{s_n}{2}\right)^2+ m(4m+1)s_n+2m^2+m\right)+1 \\
   & &b_n(s_n)=\frac{1}{2^n}s_n(s_n-4m-2)(s_n^2-4ms_n-4m) ,
 \end{eqnarray*}
thus $a_n(s_{n})\equiv  1 \ ({\rm mod} \ 4)$ and
\[
   A_n(s_{n})=3 \qquad B_n(s_{n})=1.
 \]
 Hence, the cycle $(s_n, s_n+2^{n-1})$ strongly splits and $B_n < \min \{A_n, n\} $.
Therefore, by Proposition \ref{stronglysplits}, the lift of $(s_n,
s_n+2^{n-1})$ strongly splits $B_n-1=1-1=0$ times then all lifts
strongly grow. Thus we obtain $E_1$ in Theorem
\ref{x^2+(1-4m)x-v_2(m)-odd}.

 Now we study the cycle $(2^{n-2},2^{n-2}+2^{n-1})$ at level $n\geq 4$. We have
 \begin{eqnarray*}
   & &a_n(2^{n-2})=2^{3n-2}-3(4m+1)2^{2n-3}+m(4m+1)2^{n+1}+16m^2+8m+1 \\
   & &b_n(2^{n-2})=2(2^{n-3}-2m-1)(2^{2n-6}-m2^{n-2}-m).
 \end{eqnarray*}
Thus $a_n(2^{n-2})\equiv  1 \ ({\rm mod} \ 4)$ and for each $ n >
\lfloor \frac{v_2(m)}{2} \rfloor+3 $,
 \[
   A_n(2^{n-2})=v_2(m)+3 \qquad B_n(2^{n-2})=v_2(m)+1.
 \]
 Hence, the cycle $(2^{n-2},2^{n-2}+2^{n-1})$ strongly splits and $B_n < \min \{A_n, n\} $.
Therefore, by Proposition \ref{stronglysplits}, the lift of
$(2^{n-2},2^{n-2}+2^{n-1})$ strongly splits
$B_n-1=v_2(m)+1-1=v_2(m)$ times then all lifts strongly grow. Thus
we have $E_3$.

 For each $4\leq n \leq \lfloor \frac{v_2(m)}{2} \rfloor +3 $,
 \[
   A_n(2^{n-2})=2n-3 \qquad B_n(2^{n-2})=2n-5.
 \]
 Hence, if $n>4$, then the cycle $(2^{n-2},2^{n-2}+2^{n-1})$ strongly splits and $A_n> B_n \geq n $.
Therefore, the lift of $(2^{n-2},2^{n-2}+2^{n-1})$ strongly splits
at least $n-1$ times. But except this we do not obtain any further
more information. Thus Proposition \ref{stronglysplits} is not
sufficient for us. Now we do some calculations directly.

For any point $2^{n-2}+t2^{n-1} \in 2^{n-2}+2^{n-1}\mathbb{Z}_2$,
with $t\in \mathbb{Z}_2 $, we have
\begin{eqnarray}\label{Theta}
f^2(2^{n-2}+t2^{n-1})-(2^{n-2}+t2^{n-1})=2^{n+1}(1+2t)\cdot \Theta,
\end{eqnarray}
where
\begin{eqnarray*}
  \Theta:= (2^{n-3}+t2^{n-2}-2m-1)\left((2^{n-3}+t2^{n-2})^2-m(2^{n-2}+t2^{n-1})-m
\right).
\end{eqnarray*}

 Since $4\leq n \leq \lfloor \frac{v_2(m)}{2} \rfloor +3 $, we
have $v_2(\Theta)=2n-6 $, and
\begin{eqnarray*}
f^2(2^{n-2}+t2^{n-1})-(2^{n-2}+t2^{n-1})
&\equiv&  0 \ ({\rm mod} \ 2^{3n-5})\\
&\not\equiv&  0 \ ({\rm mod} \ 2^{3n-4})
\end{eqnarray*}
Thus the cycles grow at level $3n-5$. By Corollary
\ref{grow-two-times}, the cycles grow always. Therefore we obtain
the part $E_2$ which completes the proof. \qed

\medskip
\textit{Proof of Theorem \ref{x^2+(1-4m)x-v_2(m)-even}.} Let
$f(x)=x^2+(-1-4m)x$ with $v_2(m) \in 2\mathbb{N}^*$. We see that
there are two fixed points $0$ and $4m+2$, and $1+2\mathbb{Z}_2 $ is
mapped into $2\mathbb{Z}_2 $.

As the proof of Theorem \ref{x^2+(1-4m)x-v_2(m)-odd}, we study two
$2$-cycles: $(2^{n-2},2^{n-2}+2^{n-1})$ and $(s_n, s_n+2^{n-1})$ at
level $n$, where $s_n\equiv 4m +2+ 2^{n-2} \ ({\rm mod} \ 2^{n-1})$
and $0 \leq s_n < 2^{n-1} $. The existence of $E_1,E_2,E_3 $ are the
same as that of Theorem \ref{x^2+(1-4m)x-v_2(m)-odd}.


Consider $2^{n-2}+2^{n-1}\mathbb{Z}_2$ with $ n = \frac{v_2(m)}{2}+3
$. We are going to study the cycle $(2^{n-2},2^{n-2}+2^{n-1})$ at
level $n$. We study the points $2^{n-2}+t2^{n-1} \in
2^{n-2}+2^{n-1}\mathbb{Z}_2$, with $t\in \mathbb{Z}_2$. With the
same calculation in the proof of Theorem
\ref{x^2+(1-4m)x-v_2(m)-odd}, we have the same equation
(\ref{Theta}). To continue the proof, we will distinguish two cases:
$v_2(m)=2 $ and $v_2(m)\geq 4 $.

If $v_2(m)=2 $, then $n={v_2(m)}/{2}+3=4 $ and
\begin{eqnarray*}
\Theta= (4t-2m+1)[(4-m)+16(t+t^2)-4m(1+2t)].
\end{eqnarray*}
Thus if $v_2(m-4)=3$, then $v_2(\Theta)=3 $ and
\begin{eqnarray*}
f^2(2^{n-2}+t2^{n-1})-(2^{n-2}+t2^{n-1})
&\equiv&  0 \ ({\rm mod} \ 2^{8})\\
&\not\equiv&  0 \ ({\rm mod} \ 2^{9}).
\end{eqnarray*}
 If $v_2(m-4)\geq 5 $, then $v_2(\Theta)=4 $ and
\begin{eqnarray*}
f^2(2^{n-2}+t2^{n-1})-(2^{n-2}+t2^{n-1})
&\equiv&  0 \ ({\rm mod} \ 2^{9})\\
&\not\equiv&  0 \ ({\rm mod} \ 2^{10}).
\end{eqnarray*}
Hence we will obtain (1) and (2) of Theorem
\ref{x^2+(1-4m)x-v_2(m)-even}.

Since $f^2(x)-x=x(x-4m-2)(x^2-4mx-4m) $, $f $ has $2$-periodic orbit
if and only if $x^2-4mx-4m=0 $ has solutions different to $0$ and
$4m+2$ in $\mathbb{Z}_2 $. But $x^2-4mx-4mx=0 $ has solution $0$ or
$4m+2$ only if $m=0$ or $m=-1 $. Thus for the case $v_2(m)\in
\mathbb{N}^* $, $f$ has $2$-periodic orbit if and if only
$\vartriangle:=16m^2+16m $ has square roots in $\mathbb{Z}_2 $. By
the standard argument in number theory (see \cite{Serre}, p.18),
this is equivalent to $2^{-v_2(m)}m(m+1)\equiv 1 \ ({\rm mod} \ 8)
$. By some basic calculations it is then equivalent to $v_2(m-4)=4$.
This is nothing but the rest case we need to study. Thus for
$v_2(m-4)=4 $, there exists a $2$-periodic orbit.

From the equation $x^2-4mx-4m=0 $, the periodic point can be written
as
\begin{eqnarray*}
x_1=4\left(\frac{m}{2}+ \sqrt{\frac{m(m+1)}{4}}\right), \quad
x_2=4\left(\frac{m}{2}-\sqrt{\frac{m(m+1)}{4}}\right).
\end{eqnarray*}
Recall that we are concerned with $2^{n-2}+2^{n-1}\mathbb{Z}_2 \ (n=
4) $ which is the union of two balls $2^{n-2}+2^{n}\mathbb{Z}_2 $
and $2^{n-2}+2^{n-1}+2^n\mathbb{Z}_2 $, and we are studying the
cycle $(2^{n-2},2^{n-2}+2^{n-1}) $ at level $n=4 $. Thus we have
$x_1\equiv 4 \ ({\rm mod} \ 16)$ and $x_2\equiv 12 \ ({\rm mod} \
16)$.

For each $k\geq 5$, we consider the union of the
 two balls $(x_1 + 2^{k-1}+ 2^k \mathbb{Z}_2)\cup (x_2 + 2^{k-1}+ 2^k
 \mathbb{Z}_2)$. We study the cycle $(s_1,s_2) $ where $s_1\equiv x_1+2^{k-1} \ ({\rm mod} \
 2^k)$, $s_2\equiv x_2+2^{k-1} \ ({\rm mod} \ 2^k)$ and $0\leq s_1,s_2< 2^k $.
For every point $x_1 + 2^{k-1} + t2^k \in x_1 + 2^{k-1}+ 2^k
\mathbb{Z}_2, \ (t\in \mathbb{Z}_2) $, we have
\begin{equation}\label{calculation312}
  \begin{split}
&f^2(x_1 + 2^{k-1} + t2^k)-(x_1 + 2^{k-1} + t2^k)\\
=&2^3\left(\frac{x_1}{4}+2^{k-3}+t2^{k-2}\right)\left(\frac{x_1}{2}+2^{k-2}+t2^{k-1}-2m-1\right)\cdot
\Phi,
  \end{split}
\end{equation}
where
\begin{eqnarray*}
  \Phi:= 2x_1(2^{k-1}+t2^k)+(2^{k-1}+t2^k)^2-4m(2^{k-1}+t2^k).
\end{eqnarray*}

 Here we have used the property that $x_1 $ is a solution of
the equation $x^2-4mx-4m=0 $.

Since $v_2(m)=2 $ and $v_2(x_1)=2 $, we get $v_2(\Phi)=k+2 $. Thus
\begin{eqnarray*}
f^2(x_1 + 2^{k-1} + t2^k)-(x_1 + 2^{k-1} + t2^k)
&\equiv&  0 \ ({\rm mod} \ 2^{k+5})\\
&\not\equiv&  0 \ ({\rm mod} \ 2^{k+6}).
\end{eqnarray*}
Hence we have (3).

Now we are left to treat the case $v_2(m)\geq 4 $. In this case the
equation $x^2-4mx-4m=0 $ admits solutions if and only if
$v_2(m-2^{v_2(m)}) \geq v_2(m)+3. $

We still consider $2^{n-2}+2^{n-1}\mathbb{Z}_2$ with $ n =
\frac{v_2(m)}{2}+3 $. If $v_2(m-2^{v_2(m)}) < v_2(m)+3 $, then
$v_2(\Theta)=v_2(m-2^{v_2(m)}) $ and for any $t\in \mathbb{Z}_2 $
\begin{eqnarray*}
f^2(2^{n-2}+t2^{n-1})-(2^{n-2}+t2^{n-1})
&\equiv&  0 \ ({\rm mod} \ 2^{v_2(m-2^{v_2(m)})+n+1})\\
&\not\equiv&  0 \ ({\rm mod} \ 2^{v_2(m-2^{v_2(m)})+n+2}).
\end{eqnarray*}
Then we will obtain (4).

       If $v_2(m-2^{v_2(m)}) \geq v_2(m)+3 $, then $2^{n-2}+2^{n-1}\mathbb{Z}_2$ consists of a $2$-periodic
       orbit:
       \begin{eqnarray*}
 x'_1=2^{\frac{v_2(m)}{2}+1}\left(2^{-\frac{v_2(m)}{2}}m+ \sqrt{2^{-v_2(m)}m(m+1)}\right),
 \\
 x'_2=2^{\frac{v_2(m)}{2}+1}\left(2^{-\frac{v_2(m)}{2}}m-
 \sqrt{2^{-v_2(m)}m(m+1)}\right).
       \end{eqnarray*}

For each $k\geq \frac{v_2(m)}{2}+4$, we consider
   $(x'_1 + 2^{k-1}+ 2^k \mathbb{Z}_2)\cup (x'_2 + 2^{k-1}+ 2^k
   \mathbb{Z}_2)$.
For every point $x'_1 + 2^{k-1} + t2^k \in x'_1 + 2^{k-1}+ 2^k
\mathbb{Z}_2, \ (t\in \mathbb{Z}_2) $, we have the same calculation
as (\ref{calculation312}). Since $k\geq \frac{v_2(m)}{2}+4$ and
$v_2(x'_1)=\frac{v_2(m)}{2}+1 $, we get
$v_2(\frac{x'_1}{4}+2^{k-3}+t2^{k-2})=\frac{v_2(m)}{2}-2 $ and
$v_2(\Phi)=\frac{v_2(m)}{2}+k $.
 Thus
\begin{eqnarray*}
f^2(x'_1 + 2^{k-1} + t2^k)-(x'_1 + 2^{k-1} + t2^k)
&\equiv&  0 \ ({\rm mod} \ 2^{v_2(m)+k+1})\\
&\not\equiv&  0 \ ({\rm mod} \ 2^{v_2(m)+k+2}).
\end{eqnarray*}
Hence we have (5). This completes the proof. \qed

\end{document}